\documentclass[11pt,reqno]{amsart}

\usepackage{amsmath}
\usepackage{amsthm}
\usepackage{amssymb}
\usepackage{amsfonts}
\usepackage{hyperref}
\usepackage[norsk,nameinlink]{cleveref}
\usepackage{xcolor}
\usepackage{tikz}
\usepackage{tikz-cd}
\usepackage{xcolor}
\usepackage{enumerate}
\usepackage{float}
\usepackage{stmaryrd}
\usepackage{enumitem}

\newcommand{\quot}{/\kern-0.2em/}

\definecolor{DarkRed}{RGB}{200,20,20}
\definecolor{DarkGreen}{RGB}{0,120,0}
\definecolor{SkyBlue}{rgb}{0.16, 0.32, 0.75}
\hypersetup{
    colorlinks=true,
    linkcolor=DarkGreen, 
    linkbordercolor=DarkGreen,
    citecolor=DarkGreen,
    }

\usepackage[left=1.5in, right=1.5in, top=1.4in, bottom=1.4in
]{geometry}

\numberwithin{equation}{section}

\makeatletter
\newcommand*{\rom}[1]{\expandafter\@slowromancap\romannumeral #1@}
\makeatother

\DeclareMathOperator{\M}{\mathbf{M}}

\renewcommand{\epsilon}{\varepsilon}
\newcommand{\Ms}{\overline{\M}^{\mathrm{s}}}
\newcommand{\ChQ}{\sslash_{\text{\tiny Ch}}}
\newcommand{\MBB}{\overline{\M}^{\scriptscriptstyle\mathrm{BB}}}
\DeclareMathOperator{\Hom}{Hom}
\DeclareMathOperator{\Ext}{Ext}
\DeclareMathOperator{\Gr}{Gr}
\DeclareMathOperator{\Bl}{Bl}
\DeclareMathOperator{\f}{f}
\DeclareMathOperator{\Stab}{Stab}
\DeclareMathOperator{\PGL}{PGL}
\DeclareMathOperator{\GL}{GL}
\DeclareMathOperator{\Aut}{Aut}
\DeclareMathOperator{\Spec}{Spec}
\DeclareMathOperator{\Def}{Def}
\DeclareMathOperator{\prim}{prim}
\DeclareMathOperator{\even}{even}
\DeclareMathOperator{\odd}{odd}

\DeclareMathOperator{\Pic}{Pic}
\DeclareMathOperator{\diag}{diag}
\DeclareMathOperator{\oM}{\overline{\M}}
\DeclareMathOperator{\T}{\mathcal{T}}
\newcommand*{\sExt}{\mathcal{E}\kern -.5pt xt}
\newcommand*{\sHom}{\mathcal{H}\kern -.5pt om}

\theoremstyle{plain}
\usepackage{aliascnt}

\newtheorem{theorem}{Theorem}[section]

\newaliascnt{corollary}{theorem}
\newtheorem{corollary}[corollary]{Corollary}
\aliascntresetthe{corollary}

\newaliascnt{proposition}{theorem}
\newtheorem{proposition}[proposition]{Proposition}
\aliascntresetthe{proposition}

\newaliascnt{lemma}{theorem}
\newtheorem{lemma}[lemma]{Lemma}
\aliascntresetthe{lemma}

\theoremstyle{definition}

\newtheorem*{acknowledgements}{Acknowledgements}

\newaliascnt{conjecture}{theorem}

\aliascntresetthe{conjecture}

\newaliascnt{definition}{theorem}
\newtheorem{definition}[definition]{Definition}
\aliascntresetthe{definition}

\newaliascnt{notation}{theorem}
\newtheorem{notation}[notation]{Notation}
\aliascntresetthe{notation}

\newaliascnt{construction}{theorem}

\aliascntresetthe{construction}

\newaliascnt{remark}{theorem}
\newtheorem{remark}[remark]{Remark}
\aliascntresetthe{remark}

\newaliascnt{question}{theorem}

\aliascntresetthe{question}

\newaliascnt{example}{theorem}
\newtheorem{example}[example]{Example}
\aliascntresetthe{example}

\crefname{lemma}{Lemma}{Lemmas}
\crefname{notation}{Notation}{Notation}
\crefname{question}{Question}{Questions}
\crefname{construction}{Construction}{Constructions}
\crefname{table}{Table}{Tables}
\crefname{example}{Example}{Examples}
\crefname{corollary}{Corollary}{Corollaries}
\crefname{remark}{Remark}{Remarks}
\crefname{theorem}{Theorem}{Theorems}
\crefname{definition}{Definition}{Definition}
\crefname{proposition}{Proposition}{propositions}
\crefname{figure}{Figure}{figures}
\crefname{table}{Table}{tables}
\crefname{conjecture}{Conjecture}{conjectures}
\crefformat{equation}{\color{DarkGreen}(#2#1#3)}
\crefname{enumi}{}{items}
\crefname{section}{Section}{Section}
\crefname{subsection}{Section}{Section}
\creflabelformat{enumi}{\textcolor{DarkGreen}{#2#1#3}}

\makeatletter
\renewcommand{\l@subsection}{
  \@tocline{2}{0.06em}{2.3em}{}{}}
\makeatother

\newcommand{\Adresses}{{
 \bigskip
 \footnotesize
 
Fang, Hanlong.
\textsc{Peking University, Beijing, China.} \par\nopagebreak 
\textit{Email address:} \texttt{hanlongfang@math.pku.edu.cn}

\medskip

Nguyen, Bin.
\textsc{Quy Nhon University, Quy Nhon, Vietnam.} \par\nopagebreak 
\textit{Email address:} \texttt{nguyenbin@qnu.edu.vn}

\medskip

Wu, Xian.
\textsc{University of Vienna, Vienna, Austria.} \par\nopagebreak 
\textit{Email address:} \texttt{xianwu.ag@gmail.com}, \quad \texttt{xian.wu@univie.ac.at}

\medskip

Zhang, Zheng.
\textsc{ShanghaiTech University, Shanghai, China} \par\nopagebreak 
\textit{Email address:} \texttt{zhangzheng@shanghaitech.edu.cn}
}}

\title
[Moduli of Persson Surfaces]
{Moduli of Persson Surfaces:\\
The Compactification via KSBA Stable Pairs\\ 
and a Generic Global Torelli Type Theorem}
\author{Hanlong Fang, Bin Nguyen, Xian Wu and Zheng Zhang}
\date{}


\begin{document}
\begin{abstract}
    We study a family of canonically polarized surfaces introduced by Persson, which arise as Galois $G=(\mathbb{Z}/2\mathbb{Z})^4$-covers of $\mathbf{P}^2$ branched along eight general lines. For this family, we construct the compactified moduli space and explicitly describe the stable degenerations in the sense of Koll\'ar, Shepherd-Barron, and Alexeev (KSBA) via stable pairs of weighted hyperplane arrangements. By computing the $\mathbb{Q}$-Gorenstein obstructions and using the KSBA wall crossings, we show that the resulting compactified moduli stack is smooth. Furthermore, we establish a generic global Torelli type result: up to two possibilities, a generic smooth Persson surface can be recovered from the Hodge structure on the anti-invariant part of the second cohomology of its \'etale double cover, together with the associated $\widetilde{G}=(\mathbb{Z}/2\mathbb{Z})^5$-action.
\end{abstract}

\maketitle
\tableofcontents

\section{Introduction}

For moduli spaces of smooth projective varieties, possibly with boundary divisors, there are two central problems.
\begin{enumerate}[wide]
    \item Modular compactification: one looks for a natural compactification of the moduli space by adding geometric meaningful boundary objects. The classical successful constructions are by Deligne--Mumford, Grothendieck and Knudsen for the compactification of the moduli of (marked) genus $g$ curves in \cite{delignemumford1969,Knudsen1983}. For higher dimensional varieties or pairs, the program was initiated in Koll\'ar--Shepherd-Barron \cite{kollar1988threefolds} and Alexeev \cite{Alexeev+1996+1+22}, and is now known as the \emph{KSBA compactification}. The construction of KSBA compactified moduli spaces is notoriously difficult and no general approach is currently known.
    
    \item Torelli type problem: one aims to understand the moduli via Hodge structures and period maps, relating the geometric objects to their Hodge structures. Classical results include the global Torelli theorems for principally polarized abelian varieties, K3 surfaces and Enriques surfaces.
\end{enumerate}

In this paper, we will study the moduli of \emph{Persson surfaces} from the two viewpoints above. Persson in \cite{persson78double} constructed a surface of general type $X$ with numerical invariants $p_g=3$, $q=0$, $K^2=16$ and $\pi_1(X)=(\mathbb{Z}/2\mathbb{Z})^3$ (see \cref{Lem:Picard}) such that $|K_X|$ gives a map $\pi\colon X\to\mathbf{P}^2$ of degree $16$. He proposed a construction of $X$ as a double cover $X\to W$, where $W$ is a certain Campedelli surface with $p_g=q=0$, $K^2=2$ and $\pi_1(W)=(\mathbb{Z}/2\mathbb{Z})^3$, as studied by Peters \cite{peters1976}. For a long time, Persson surfaces provided the only known examples of surfaces of general type whose canonical maps have degree greater than $9$.

Mendes Lopes and Pardini in \cite{lopes2021degree} have given a description of such $X$ as a Galois $G=(\mathbb{Z}/2\mathbb{Z})^4$-cover of $\mathbf{P}^2$ branched along eight general lines $D=\sum_{i=1}^8 D_i$. The cover is determined by Pardini's building data in \cite{pardini1991abelian} completely, which consist of eight lines labeled by some $g\in G$ and $16$ line bundles labeled by the characters; see \cref{Def:Persson}.

By generalizing the abelian cover theory for normal varieties in Pardini \cite{pardini1991abelian} to non-normal ones, Alexeev--Pardini in \cite{alexeev2012non} proved that the (KSBA) stability of degeneration of an abelian cover is equivalent to the stability of the base pair with coefficients of divisors determined by $G$. For the base pairs $(\mathbf{P}^2,\frac12\sum_{i=1}^8 D_i)$, which can be viewed as a weighted hyperplane arrangement, the explicit descriptions of the  stable degenerations were given by Alexeev in \cite{alexeev2015moduli} by using matroid tilings and geometric invariant theory (GIT).

Before stating our first main result, we recall the following definition.

\begin{definition}
    We denote by $\M(\mathbf{P}^{d-1},n)$  the moduli space of pairs $(\mathbf{P}^{d-1},D=\sum D_i)$ such that $D$ consists of $n$ hyperplanes $D_1,\ldots,D_n$ in general linear position (equivalently, simply normal crossed).
    For any \emph{weight vector} $\mathbf{b}=(b_1,\ldots,b_n)\in(0,1]^n\cap\mathbb{Q}^n$, 
    we denote by $\oM_{\mathbf{b}}(\mathbf{P}^{d-1},n)$ the main irreducible component of the KSBA compactification of $\mathbf{b}$-weighted hyperplane arrangement $(\mathbf{P}^{d-1},\mathbf{b}D:=\sum b_iD_i)$.
\end{definition} 

\begin{remark}\label{Rmk:compare_with_Kollar}
Note that $\oM_{\mathbf{b}}(\mathbf{P}^{d-1},n)$ is the irreducible component of $\mathrm{SP}(\mathbf{b},d-1,(\sum b_i-d)^{d-1})$ in \cite[Theorem~8.1]{kollar2023families} containing $\M(\mathbf{P}^{d-1},n)$ as an open subset.
\end{remark}

\begin{theorem}[{=\cref{Thm:Persson_KSBA,Thm:smoothness_stack} and  \cref{Prop:degeneration_cover}}]\label{Thm:Persson_KSBA_in_intro}
   Let $\Ms$ be the coarse moduli of stable Persson surfaces associated to the KSBA (or KSB following terminologies in \cite{kollar2023families} as the boundary is empty) compactified Deligne--Mumford stack $\overline{\mathcal{M}}^{\mathrm{s}}$. Let $\mathbf{b}=(\frac12,\ldots,\frac12)$. Then we have:
\begin{enumerate}[label=\textup{(\arabic*)}]
    \item $\Ms\cong\oM_{\mathbf{b}}(\mathbf{P}^2,8)/G_{\Stab}$, where the group $G_{\Stab}$ is a isomorphic to the affine transformation group $\mathrm{Aff}(\mathbb{F}_2,3)=(\mathbb{Z}/2\mathbb{Z})^3\rtimes\GL(\mathbb{F}_2,3)$. Thus, the moduli $\Ms$ is unirational.
    
    \item The Deligne--Mumford stack $\overline{\mathcal{M}}^{\mathrm{s}}$ is smooth.
\end{enumerate}
    For stable degenerations of Persson surfaces, we have:
\begin{enumerate}[resume, label=\textup{(\arabic*)}]
    \item The stable degenerations of Persson surfaces are $G$-covers of the three pairs as shown in \cref{Fig:degenerationofP2} with eight branch broken lines, which are
    stable degenerations of $(\mathbf{P}^2,\mathbf{b} D)$. In particular, the ambient variety $\mathbf{P}^2$ may degenerate into three types: type $0$, $\mathbf{P}^2$; type \rom{1}, $\mathbf{P}^2\cup\mathbf{F}_1$, where $\mathbf{F}_1$ denotes the first Hirzebruch surface; and type \rom{2}, $\mathbf{P}^2\cup(\mathbf{P}^1\times\mathbf{P}^1)\cup\mathbf{P}^2$.

    \item For generic type \rom{1} degenerations $\mathbf{P}^2\cup\mathbf{F}_1$, the $G$-cover of $\mathbf{P}^2$ is a K3 surface with eight $A_1$-singularities, and the $G$-cover of $\mathbf{F}_1$ is a elliptic surface with eight $A_1$-singularities. For generic type \rom{2} degenerations $\mathbf{P}^2\cup(\mathbf{P}^1\times\mathbf{P}^1)\cup\mathbf{P}^2$, the $G$-cover of each $\mathbf{P}^2$ is a K3 surface with four $A_3$-singularities, and the $G$-cover of $\mathbf{P}^1\times\mathbf{P}^1$ is a smooth K3 surface. Here by generic we mean that branch lines are simply normal crossed on each component.  Irreducible components are glued along elliptic curves arising as covers of the double lines $\mathbf{P}^1$ (the white lines) in \cref{Fig:degenerationofP2}.
\end{enumerate}
\end{theorem}

\begin{center}
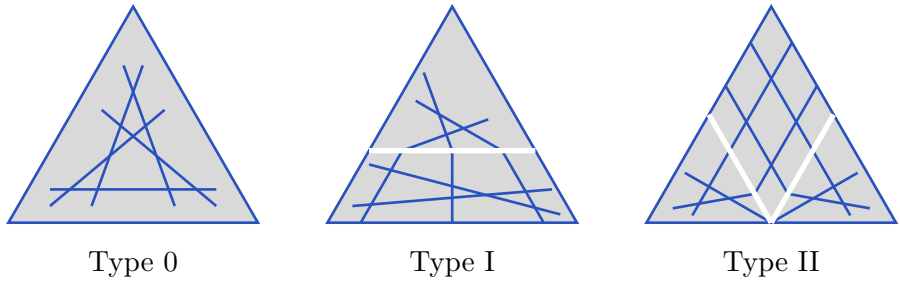
\begin{figure}[htpb]
\begin{minipage}{0.25\textwidth}
\begin{tikzpicture}[scale=1.1, line width=1pt]
        \filldraw [gray, opacity=0.3] (0, 0) -- (3, 0) -- ({3*cos(60)}, {3*sin(60)}) -- cycle;
        \draw [SkyBlue, line width=1pt] (0, 0) -- (3, 0) -- ({3*cos(60)}, {3*sin(60)}) -- cycle;
        \draw [SkyBlue, line width=1pt] (1,0.2) -- ++(70:1.8);
        \draw [SkyBlue, line width=1pt] (0.5,0.2) -- ++(40:1.8);
        \draw [SkyBlue, line width=1pt] (2,0.2) -- ++(110:1.8);
        \draw [SkyBlue, line width=1pt] (2.5,0.2) -- ++(140:1.8);
        \draw [SkyBlue, line width=1pt] (0.5,0.4) -- (2.5,0.4);
        \node at (1.5,-0.5) {Type 0};
\end{tikzpicture}
\end{minipage}\hspace{0.6cm}
\begin{minipage}{0.25\textwidth}
\begin{tikzpicture}[scale=1.1, line width=1pt]
        \filldraw [gray, opacity=0.3] (0, 0) -- (3, 0) -- ({3*cos(60)}, {3*sin(60)}) -- cycle;
        \draw [SkyBlue, line width=1pt] (0, 0) -- (3, 0) -- ({3*cos(60)}, {3*sin(60)}) -- cycle;
        \draw [SkyBlue, line width=1pt] (0.4,0) -- ++(60:1) -- ++(20:1.1);
        \draw [SkyBlue, line width=1pt] (2.6,0) -- ++(120:1) -- ++(150:1.2);
        \draw [SkyBlue, line width=1pt] (1.5,0) -- ++ (90:sqrt(3/4) -- ++(110:1);
        \draw [SkyBlue, line width=1pt] (0.3,0.2) -- (2.7,0.4);
        \draw [SkyBlue, line width=1pt] (0.5,0.7) -- (2.8,0.1);
        \draw [white, line width=2pt] (60:1) -- ++(2,0);
        \node at (1.5,-0.5) {Type \rom{1}};
\end{tikzpicture}
\end{minipage}\hspace{0.6cm}
\begin{minipage}{0.25\textwidth}
\begin{tikzpicture}[scale=1.1, line width=1pt]
        \filldraw [gray, opacity=0.3] (0, 0) -- (3, 0) -- ({3*cos(60)}, {3*sin(60)}) -- cycle;
        \draw[SkyBlue, line width=1pt] (0, 0) -- (3, 0) -- ({3*cos(60)}, {3*sin(60)}) -- cycle;
        \coordinate (A) at (1.5,0);
        \coordinate (a) at ([shift={(-0.2,0.2*1.732)}] A);
        \draw[SkyBlue, line width=1pt] (a) -- ++(60:1.5);
        \coordinate (b) at ([shift={(-0.5,0.5*1.732)}] A);
        \draw[SkyBlue, line width=1pt] (b) -- ++(60:1.5);
        \coordinate (e) at ([shift={(0.2,0.2*1.732)}] A);
        \draw[SkyBlue, line width=1pt] (e) -- ++(120:1.5);
        \coordinate (f) at ([shift={(0.5,0.5*1.732)}] A);
        \draw[SkyBlue, line width=1pt] (f) -- ++(120:1.5);
        \draw[SkyBlue, line width=1pt] (a) -- ++(-170:1);
        \draw[SkyBlue, line width=1pt] (b) -- ++(-120:0.9);
        \draw[SkyBlue, line width=1pt] (e) -- ++(-10:1);
        \draw[SkyBlue, line width=1pt] (f) -- ++(-60:0.9);
        \draw[SkyBlue, line width=1pt] (A) -- ++(150:1.2);
        \draw[SkyBlue, line width=1pt] (A) -- ++(30:1.2);
        \draw[line width=2pt, white] (A) -- ++(120:1.5);
        \draw[line width=2pt, white] (A) -- ++(60:1.5);
        \node at (1.5,-0.5) {Type \rom{2}};
\end{tikzpicture}
\end{minipage}
\caption{Generic stable degenerations of $(\mathbf{P}^2,\mathbf{b}\sum_{i=1}^8D_i)$.}\label{Fig:degenerationofP2}
\end{figure}
\end{center}

The complete descriptions of the compactified moduli and all stable degenerations were given for Campedelli surfaces in \cite{alexeev2024explicit}, which are $(\mathbb{Z}/2\mathbb{Z})^3$-covers of $\mathbf{P}^2$ branched along seven general lines (see \cref{Rmk:surfaces_in_the_middle}\cref{Item:1_rmk_surfaces_in_middle}). 
In particular, only type $0$ degenerations occur, that is, in the language of polytopes, the $\mathbf{b}$-cut admits only trivial matroid tilings (see \cref{Def:b-cut}). Hence the smoothness of its KSBA moduli stack is immediate by Hacking's strategy in \cite{hacking2004compact}.

In general, to prove the smoothness of the KSBA moduli stacks, 
the computation of $\mathbb{Q}$-Gorenstein obstruction spaces could be impractical,
since the variety parts of the pairs can be non-$\mathbb{Q}$-Gorenstein. 
In our situation, we observe that the wall crossing isomorphism between different moduli of weighted hyperplane arrangements given in \cite[Section~5.5]{alexeev2015moduli} can be applied, which reduces the computation of the original obstructions to that of a $\mathbb{Q}$-Gorenstein variety.
Also by the KSBA wall crossings, we show that there is a morphism from the $\mathbf{b}$-weighted KSBA compactification to the GIT compactification in \cref{Prop:Morphism_KSBA_to_GIT}.

In \cite{oudomphengthesis}, Oudompheng constructed the Baily--Borel compactification of a special type of Enriques surfaces, so-called $D_{1,6}$-polarized Enriques surfaces (see \cref{Def:d16_pol_Enriques}). In the middle of our $G$-cover, there are $28$ such $D_{1,6}$-polarized Enriques surfaces appearing. Applying Oudompheng's work, we define the \emph{minimal compactification} of the moduli of Persson surfaces (see \cref{Def:BB_for_Persson}), denoted by $\MBB$. One naturally expects a morphism from the KSBA moduli to $\MBB$, for example, as \cite[Theorem~1.4]{gallardo2021geometric} for marked cubic surfaces and \cite[Theorem~3.17]{alexeev2023stable} for K3 surfaces. The following proposition confirms this expectation.
\begin{proposition}[{=\cref{Prop:KSBA_to_BB}}]\label{Prop:KSBA_to_BB_in_intro}
    There is a surjective morphism from the KSBA compactification $\Ms$ to the minimal compactification $\MBB$.
\end{proposition}

For surfaces of general types, due to Todorov's famous counterexample in \cite{todorov1981construction}, global Torelli fails, that is, the surface cannot be determined by its Hodge structure even up to finitely many possibilities. In this paper, we identify an appropriate Hodge structure for each Persson surface that determines the surface up to two possibilities. The construction of our period map and the approach to the generic global Torelli problem are inspired by the following three ingredients.
\begin{enumerate}
    \item Global Torelli for Enriques surfaces: the correct Hodge structures determining the geometry are the orthogonal part of the second cohomology of the \'etale double cover K3 surfaces; see Horikawa \cite{horikawa1977torellienriques}.

    \item Global Torelli for Prym varieties: the works Friedman--Smith \cite{friedmansmith1982Inventiones} and Kanev \cite{kanev1983prymtorelli} help us to specify the objects parametrized by the source moduli of the period map; in particular, a choice of a $2$-torsion line bundle is needed.

    \item Global Torelli for a special family of Horikawa surfaces: Pearlstein--Zhang in \cite{pearlstein2019generic} proved the generic global Torelli for a family of bidouble covers of $\mathbf{P}^2$ (see \cref{Example:special_horikawa}), with respect to Pardini's $(\mathbb{Z}/2\mathbb{Z})^2$-period map. The key is the two intermediate K3 surfaces appearing in the bidouble cover.
\end{enumerate}

Specifically, we consider the \'etale double covers $Z_L\to X$ induced by a $2$-torsion line bundle $L$ in $\Pic(X)[2]\cong(\mathbb{Z}/2\mathbb{Z})^3$ (see \cref{Lem:Picard}). The surface $Z_L$ is an abelian cover of $\mathbf{P}^2$ with Galois group $\widetilde{G}=(\mathbb{Z}/2\mathbb{Z})^5$, branched along the same eight lines as $X$, with its building data explicitly given in \cref{Prop:composition_is_abelian}. The associated moduli space $\widetilde\M$ is a $(\mathbb{Z}/7\mathbb{Z})$-cover of the moduli of smooth Persson surfaces in view of the seven $2$-torsion line bundles that intervene; see \cref{Prop:tM_MP2_8_quotient}. The primary reason for introducing $Z_L$ lies in the fact that the cover $Z_L\to\mathbf{P}^2$ factors on four K3 surfaces, which do not appear as intermediate surfaces for $X\to \mathbf{P}^2$; see \cref{Rmk:ZL_eigen_interpretation}\ref{Item:1_Rmk:ZL_eigen_interpretation}.

Therefore, we focus on the anti-invariant component of $H^2(Z_L,\mathbb{Q})$ with respect to the covering involution $Z_L\to X$, denoted by $H^2(X,\mathbb{Q})_L^\perp$, together with the $\widetilde{G}$-action. This encodes the Hodge-theoretic data required to distinguish between distinct Persson surfaces. More precisely, there exists a $24$-dimensional subrepresentation $V_\mathbb{Q}$ of $H^2(X,\mathbb{Q})_L^\perp$, denoted by 
\[
    \rho\colon\widetilde{G}\rightarrow\GL(V_\mathbb{Q})
\]
that arises from the transcendental parts of the aforementioned four intermediate K3 surfaces between $Z_L$ and $\mathbf{P}^2$, and completely determines the Hodge structure on $H^2(X,\mathbb{C})_L^\perp$. Applying the global Torelli theorem for K3 surfaces, we show that these four K3 surfaces generically determine the Persson surface up to two possibilities. Here, the genericity condition stems from recovering the overall eight-line arrangement from four $6$-line sub-arrangements (see  \cref{Lem:4times6_imply_8}). Finally, this failure of injectivity (the twofold ambiguity) is closely related to the group $G_{\Stab}\cong \mathrm{Aff}(\mathbb{F}_2,3)$ discussed in \cref{Thm:Persson_KSBA_in_intro} (see also \cref{Rmk:notinjective}).

We now state our second main result using the $\widetilde{G}$-period map $\mathcal{P}\colon\widetilde{\M}\rightarrow\mathcal{D}^\rho/\Gamma_\rho$ for a specific discrete subgroup $\Gamma_\rho$; see \cref{Def:G-period_map}, also \cref{Notation:MSY} and  \cref{Eqn:tilde_Gamma_T}, \cref{Eqn:gamma_rho} for the notation. Such period maps were originally introduced by Pardini \cite{pardini1998period} when investigating the infinitesimal Torelli problem for abelian covers; here we follow the presentation given by  Dolgachev--Kondo \cite{dolgachev2007moduli}. Such period maps have been successfully employed to describe various moduli spaces; for example, see \cite{kondo2000genus3,act2002cubicsurf,act2011cubic3,ls2007cubic3, lpz2018eckardt}.
\begin{theorem}[{=\cref{Thm:degree_2_periodmap}}]\label{Thm:global_Torelli_in_intro}
    The $\widetilde{G}$-period map $\mathcal{P}\colon \widetilde\M\to \mathcal{D}^\rho/\Gamma_{\rho}$ is generically of degree $2$, where the monodromy $\Gamma_{\rho}$ can be given explicitly.
\end{theorem}

The paper is organized as follows. In \cref{Sec:geometry}, we analyze the geometric structures of smooth Persson surfaces. We compute the numerical invariants and the intersection form lattice. We discuss the Picard group, and compute the dimension of the deformation space. We interpret the $G^*$-eigenspaces via intermediate surfaces. In \cref{Sec:KSBA}, we briefly review the theory in \cite{alexeev2015moduli}, then describe the KSBA moduli and degenerations, of both the $\mathbf{b}$-weighted line arrangements and Persson surfaces. We also prove the smoothness of the compactified stack and classify singularities appearing on stable Persson surfaces.
For our weighted line arrangements, we show that there is a morphism from the KSBA moduli to the standard GIT moduli. 
In \cref{Sec:oudompheng}, we define another compactified moduli of Persson surfaces, based on Oudompheng's work on the moduli of a special family of Enriques surfaces, then study its relation to the KSBA moduli.
In \cref{Sec:auxiliary_surfaces}, we construct a family of surfaces whose cohomologies encode the Hodge structures capturing the geometry of Persson surfaces. In \cref{Sec:proof_torelli}, we discuss the monodromy of the $\widetilde{G}$-period map, and prove that the $\widetilde{G}$-period map is generically of degree 2.

Throughout this paper, we work over the complex number field $\mathbb C$.

\begin{acknowledgements}
    We are grateful to Valery Alexeev for patiently explaining his work in \cite{alexeev2008weighted,alexeev2015moduli} to us. We thank Jungkai Alfred Chen for suggesting the problem of compactifying the moduli of Persson surfaces. We thank Rita Pardini for pointing out \cref{Rmk:from_Pardini} to us. We are indebted to Chris Peters for carefully reading the first draft of this paper and offering numerous helpful suggestions and feedback. We also appreciate enlightening discussions with Patricio Gallardo, Feng Hao, Gregory Pearlstein and Luca Schaffler. H.~Fang and X.~Wu are supported by National Key R\&D Program of China under Grant No.~2022YFA1006700. X.~Wu is also partially supported by Austrian Science Fund (FWF) project FW506067 of Herwig Hauser. Z.~Zhang is supported in part by NSFC grant 12201406.
\end{acknowledgements}

\section{The geometry of Persson surfaces}\label{Sec:geometry}

The abelian cover theory for algebraic varieties was developed in Pardini \cite{pardini1991abelian} and later further generalized to non-normal cases in Alexeev--Pardini \cite{alexeev2012non}.
We briefly recall the theory specialized to the group $G=(\mathbb{Z}/2\mathbb{Z})^m$ (see \cite[Example~2.1.\romannumeral3)]{pardini1991abelian}). 

Let $G$ be a finite abelian group. A \emph{$G$-cover} is a finite morphism $X\rightarrow Y$ of varieties such that the quotient map is induced by a generically faithful $G$-action. We also refer to $X$ as the $G$-cover of $Y$.
Denote $G^*=\Hom(G,\mathbb{G}_m)$ the character group of $G$. For a smooth variety $Y$, the \emph{building data} of a Galois $G$-cover consist of the following objects and relations. 
\begin{enumerate}[label=(\Roman*)]
    \item Branch divisors: there is a divisor $D_g$ (which might be empty) on $Y$ for each $g\in G$.
    
    \item Line bundles: there is a line bundle $L_{\chi}$ on $Y$ for each $\chi\in G^*$.
    
    \item \label{Item:3_fundamental_relation} Fundamental relations: $L_{\chi}\otimes L_{\chi^{\prime}}=L_{\chi+\chi^{\prime}}\otimes\mathcal{O}_Y(\sum _{\chi(g)=\chi^{\prime}(g)=-1}D_g)$. Notice that if the $2$-torsion part of the Picard group $\Pic(Y)[2]=0$, this is equivalent to $L_{\chi}^{\otimes 2}=\mathcal{O}_Y(\sum_{\chi(g)=-1}D_g)$. 
\end{enumerate}
\begin{remark}\label{Rmk:ab_cover_uniqueness}
    If $\Pic(Y)[2]=0$, then the line bundles can be uniquely solved from the fundamental relations. Thus, the cover $X$ is uniquely determined by the branch divisors $D_g$, also called the \emph{branch data}; see \cite[Theorem,~Page~191]{pardini1991abelian} and \cite[Definition~2.6]{alexeev2024explicit}. Notice that this is not saying that the $G$-cover as a variety is uniquely determined by the support of the branch divisors. For example, if $Y$ is a projective space, and each $D_g$ is either a hyperplane or empty such that $\sum D_g$ is in general linear position, then the moduli of the covers $X$ is a quotient of the moduli of the pairs $(Y,\sum D_g)$ by a finite group $G_{\Stab}$ given as follows. View the finite abelian group $G$ as a product of vector spaces over finite fields $\prod_p \mathbb{F}_p^{k_p}$, whose automorphism group is $G_{\Aut}:=\prod_p\GL(\mathbb{F}_p,k_p)$. Then $G_{\Stab}$ is the subgroup of $G_{\Aut}$ which stabilizes the set $\{g\in G\mid D_g\neq\emptyset\}$.
\end{remark}

\begin{example}\label{Example:special_horikawa}
    In \cite{pearlstein2019generic}, Pearlstein and Zhang considered a special family of Horikawa surfaces. Such surfaces are bidouble covers of $\mathbf{P}^2$ with the following building data:
\begin{enumerate}
    \item $D_{(0,0)}=\emptyset$, $D_{(1,0)}$, $D_{(0,1)}$ are two general lines, and $D_{(1,1)}$ is a smooth quintic curve;
    \item $L_{(0,0)}=\mathcal{O}_{\mathbf{P}^2}$, $L_{(1,0)}=L_{(0,1)}=\mathcal{O}_{\mathbf{P}^2}(3)$ and $L_{(1,1)}=\mathcal{O}_{\mathbf{P}^2}(1)$. 
\end{enumerate}
    They satisfy the fundamental relation \cref{Item:3_fundamental_relation}: $L_{(1,1)}^{\otimes 2}=\mathcal{O}_{\mathbf{P}^2}(D_{(1,0)}+D_{(0,1)})$, $L_{(1,0)}^{\otimes 2}=\mathcal{O}_{\mathbf{P}^2}(D_{(1,0)}+D_{(1,1)})$, and $L_{(0,1)}^{\otimes 2}=\mathcal{O}_{\mathbf{P}^2}(D_{(0,1)}+D_{(1,1)})$.
    Notice that $L_{(1,0)}$ and $L_{(0,1)}$ define two double covers of $\mathbf{P}^2$, which are K3 surfaces of degree $2$ with five $A_1$-singularities.
\end{example}

Now fix the abelian group $G=(\mathbb{Z}/2\mathbb{Z})^4$ and a character $\chi_0=(1,0,0,0)\in G^*$. The following definition is built from Pardini's \emph{almost uniform cover}, which is equivalent to Persson's original construction \cite[Example 5.8]{persson78double}.
\begin{definition}[{\cite[Example~4.6]{lopes2021degree}}]\label{Def:Persson}
A \emph{Persson surface}
\[
    X=\underline{\Spec}_{\mathcal{O}_{\mathbf{P}^2}}(\oplus_{\chi\in G^*}L_\chi^{-1})
\]
is a Galois $G$-cover $\pi\colon X\to\mathbf{P}^2$ whose building data satisfy the following conditions. 
\begin{enumerate}[wide]
    \item The branch divisor $D:=\sum_{g\in G} D_g$ is a line arrangement in general linear position, i.e., simply normal crossed, where $D_g$ is a line if $\chi_0(g)=-1$, equivalently, $g=(1,*,*,*)$; and $D_g=\emptyset$ otherwise. In particular, the branch divisor $D$ consists of eight lines.
    \item $L_{(0,0,0,0)}=\mathcal{O}_{\mathbf{P}^2}, L_{\chi_0}=\mathcal{O}_{\mathbf{P}^2}(4)$, and the rest $14$ line bundles are $L_\chi=\mathcal{O}_{\mathbf{P}^2}(2)$ for all $\chi\not\in\{\chi_0,(0,0,0,0)\}$, which are solved from the fundamental relations \cref{Item:3_fundamental_relation}.
\end{enumerate} 
\end{definition}

\begin{remark}\label{Rmk:surfaces_in_the_middle}
    Persson surfaces can be embedded in $\mathbf{P}(1^3,2^{14})$; see \cite[Section~5.8]{Fallucca_Pignatelli2024} for their defining equations. The covering map $\pi\colon X\to\mathbf{P}^2$ factors into a sequence of four double covers. For later use, we highlight several intermediate surfaces appearing in this decomposition.
\begin{enumerate}[wide, label=(\arabic*)]
    \item\label{Item:1_rmk_surfaces_in_middle} There are eight $(\mathbb{Z}/2\mathbb{Z})^3$-covers of $\mathbf{P}^2$ branched along seven lines. Such surfaces are smooth Campedelli surfaces, with numerical invariants $K^2=2,p_g=q=0$ and $h^{1,1}=8$. They correspond to subgroups $N^*\cong(\mathbb{Z}/2\mathbb{Z})^3$ of $G^*$ with $\chi_0\not\in N^*$. The stable degenerations and the compactification of the moduli space of Campedelli surfaces are given in \cite[Section~3~and~5.2]{alexeev2024explicit}. 
    
    \item\label{Item:2_rmk_surfaces_in_middle} There are $28$ $(\mathbb{Z}/2\mathbb{Z})^2$-covers of $\mathbf{P}^2$ branched along six lines. Such surfaces are Enriques surfaces with six $A_1$-singularities. Resolving these singularities, one obtains smooth Enriques surfaces, called $D_{1,6}$-polarized Enriques surfaces; see \cref{Def:d16_pol_Enriques}. There is a one-to-one correspondence between such Enriques surfaces and subgroups of $G^*$ in the set $\mathcal{N}=\{N^*<G^*\mid N^*\cong(\mathbb{Z}/2\mathbb{Z})^2\text{ and }\chi_0\not\in N^*\}$. Geometry and moduli of $D_{1,6}$-polarized Enriques surfaces are well studied in Oudompheng's thesis \cite{oudomphengthesis} via GIT and period map (Baily--Borel). See also Schaffler \cite{schaffler2022ksba} for the KSBA and Looijenga's semi-toric compactifications.
    
    \item\label{Item:3_rmk_surfaces_in_middle}  There are $14$ double covers of $\mathbf{P}^2$ branched along four lines, which are isomorphic to each other. Such surfaces are degree $2$ del Pezzo surfaces with six $A_1$-singularities. They correspond to order 2 subgroups of $G^*$ generated by $\chi\not\in\{0,\chi_0\}$. The quadruples of $g$'s labeling the divisors are given in the first three pictures of \cref{Fig:4g_for_22_dP2}, where the four $g$'s must be on a plane of $\mathbb{F}_2^3$.
    
    \item\label{Item:4_rmk_surfaces_in_middle} There is a unique double cover of $\mathbf{P}^2$ branched along eight lines. This surface is a special Horikawa surface studied in \cite[Theorem 1.6(\romannumeral 1)]{horikawa1976algebraic} with $28$ $A_1$-singularities, whose numerical invariants are $K^2=2,p_g=3,q=0$, and $h^{1,1}=38$ (after resolving the $A_1$-singularities).
    The surface corresponds to the order $2$ subgroup of $G^*$ generated by $\chi_0$. 
\end{enumerate}
\end{remark}

\begin{proposition}\label{Prop:numerical_data_of_Persson}
    Let $X$ be a smooth Persson surface. Then $K_X^2=16$, $q(X)=0$, $p_g(X)=3$ and $h^{1,1}(X)=24$. Furthermore, the lattice given by the intersection form on $H^2_{\f}(X,\mathbb{Z}):=H^2(X,\mathbb{Z})/\{\text{torsion}\}$ is isometric to $U^{\oplus 7}\oplus E_8(-1)^{\oplus 2}$.  
\end{proposition}
\begin{proof}
    According to \cite[Lemma~2.10(1)]{alexeev2024explicit}, $2K_X=2\pi^* H$. In particular, $K_X\sim_{\mathbb{Q}}\pi^*(K_{\mathbf{P}^2}+\frac12(8H))=\pi^* H$. It follows that $K_X^2=16$. By \cite[Proposition 4.1.c]{pardini1991abelian}, one has
\[
    H^0(X,\mathcal{O}_X(K_X))=\bigoplus_{\chi\in G^*}H^0(\mathbf{P}^2,\mathcal{O}_{\mathbf{P}^2}(K_{\mathbf{P}^2})\otimes L_{\chi^{-1}})\cong\mathbb{C}^3\not=0,
\]
    where the non-trivial part is contributed by $\chi_0$. This shows that $p_g(X)=3$. 
    The irregularity $q(X)=h^1(\mathcal{O}_X)=\sum_{\chi\in G^*}h^1(L_\chi^{-1})=h^1(\mathcal{O}_{\mathbf{P}^2})+14h^1(\mathcal{O}_{\mathbf{P}^2}(-2))+h^1(\mathcal{O}_{\mathbf{P}^2}(-4))=0$.  By the fact that $c_2(X)=2-4q(X)+2p_g(X)+h^{1,1}(X)$, we have $h^{1,1}(X)=24$. Next, let us determine the intersection form on $H^2_{\f}(X,\mathbb{Z})$. We claim that $K_X-\pi^*H$ is not torsion and hence $K_X=\pi^* H$. It suffices to show that $H^0(X,\mathcal{O}_X(K_X-\pi^*H))\neq 0$. By \cite[Proposition~4.1.c]{pardini1991abelian} again, one has
\[
    \pi_*\mathcal{O}_X(K_X-\pi^*H)=\bigoplus_{\chi\in G^*}\mathcal{O}_{\mathbf{P}^2}(K_{\mathbf{P}^2})\otimes L_{\chi^{-1}}\otimes\mathcal{O}_{\mathbf{P}^2}(-H).
\]
    Thus, as $L_{\chi^{-1}}=L_\chi$ for all $\chi\in G^*$ in our case, one obtains
\begin{align*}
    H^0(X,\mathcal{O}_X(K_X-\pi^*H))&\cong\bigoplus_{\chi\in G^*}H^0(\mathbf{P}^2,\mathcal{O}_{\mathbf{P}^2}(-4)\otimes L_{\chi})\\
    &=H^0(\mathbf{P}^2,\mathcal{O}_{\mathbf{P}^2}(-4)\otimes L_{\chi_0})=\mathbb{C}.
\end{align*}
    Note that $\pi\colon X\rightarrow \mathbf{P}^2$ decomposes into $\pi=\pi''\circ\pi'$, where $\pi'$ is a double cover and $\pi''$ is a $(\mathbb{Z}/2\mathbb{Z})^3$-cover. Note also that $\pi'^*(H)$ is $2$-divisible since one can choose $H$ to be a branched line. So $K_X=2(\pi'')^*(\pi'^*(H)/2)=:2H'$. By \cite[Lemma~\rom{8}.(3.1)]{barth2015compact}, $H^2_{\f}(X,\mathbb{Z})$ is an even lattice. The index of $H^2_{\f}(X,\mathbb{Z})$ is $\tau(X)=(K_X^2-2\chi_{\text{top}}(X))/3=-16$ by Hirzebruch index theorem (see for instance \cite[Theorem 3.4]{kondo_k3}), so the signature of $H^2_{\f}(X,\mathbb{Z})$ is $(7,23)$. The claim then follows from the classification of indefinite even unimodular lattices.
\qedhere 
\end{proof}

\begin{lemma}\label{Lem:Picard}
    The fundamental group and the torsion part of the Picard group of a smooth Persson surface $X$ are both isomorphic to $(\mathbb{Z}/2\mathbb{Z})^3$.
\end{lemma}
\begin{proof}
    Notice that $X$ is a double cover of a Campedelli surface $S$. Then we know $H^0(S,\mathcal{O}_S(2K_{S}))\neq 0$, so there is an irreducible, ample curve $C\in|2K_S|$ by \cite[Proposition 3.2]{persson78double}. This verifies the required condition in \cite[Proposition 3.3]{pardini1995fundamental}. The kernel of $\pi_1(X)\rightarrow\pi_1(S)$ is a quotient group of $\ker((\mathbb{Z}/2\mathbb{Z})\xrightarrow{\mathrm{id}}(\mathbb{Z}/2\mathbb{Z}))$. Thus $\pi_1(X)\cong\pi_1(S)$, which is $(\mathbb{Z}/2\mathbb{Z})^3$; for instance, see \cite[Section~3]{alexeev2024explicit}. 
    Using the universal coefficient system theorem, since $H^1(X,\mathcal{O}_X)=0$ by \cref{Prop:numerical_data_of_Persson}, $\Pic(X)$ is embedded in $\mathbb{Z}^{24}\oplus(\mathbb{Z}/2\mathbb{Z})^3$. Taking the long exact sequence of the exponential sheaf sequence, by \cref{Prop:numerical_data_of_Persson} and the computation in \cref{Cor:eigen_Persson}, one has
\[
    0=H^1(X,\mathcal{O}_X)\to H^1(X,\mathcal{O}_X^*)\to H^2(X,\underline{\mathbb{Z}})\to H^2(X,\mathcal{O}_X)\cong\mathbb{C}^3.
\]
    As $H^2(X,\mathcal{O}_X)$ is torsion free, the map from the torsion part of $H^1(X,\mathcal{O}_X^*)\cong\Pic(X)$ to the torsion part of $H^2(X,\underline{\mathbb{Z}})$ is surjective. So the torsion pat of $\Pic(X)$ is isomorphic to $(\mathbb{Z}/2\mathbb{Z})^3$.
\end{proof}

\begin{remark}\label{Rmk:torsion_L_vs_four_pairs} 
\begin{enumerate}[wide, label=(\arabic*)]
    \item\label{Item:1_Rmk:torsion_L_vs_four_pairs} There is a bijection between the seven non-trivial $2$-torsion line bundles on $X$ and the partitions of the affine space $\mathbb{F}_2^3$ into four parallel lines. The $2$-torsion line bundles in $\Pic(X)$ arise as the pullbacks of the canonical line bundles of the intermediate Enriques surfaces described in \cref{Rmk:surfaces_in_the_middle}\ref{Item:2_rmk_surfaces_in_middle}. The canonical divisor of such an Enriques surface is given by the difference of the preimages of the lines $l_i, l_i'$, $i=1,2,3$, on $\mathbf{P}^2$ under the associated bidouble cover, as illustrated in \cref{Fig:hexagon}. Geometrically, these three pairs of lines determine a partition of the eight branch lines $\{D_g \mid g\in G, D_g\neq\emptyset\}$ into four pairs, which we denote by $\{D_{g_i},D_{g_i'}\}_{i=1}^4$. Conversely, for each such partition, the corresponding $2$-torsion line bundle is given by $\mathcal{O}_X(\pi^{-1}D_{g_i}-\pi^{-1}D_{g_i'})$ for any $i=1,\ldots,4$, where $\pi^{-1}$ denotes the reduced preimage. Finally, we mention that this partition is uniquely determined by choosing the partner $D_{g_1'}$ (for which there are seven choices) of the fixed line $D_{g_1}$ corresponding to the index $g_1=(1,0,0,0)$; explicitly, the remaining indices are determined by the relation $g_i'=g_i+g_1'+(1,0,0,0)$ for $i=2,3,4$.
    
    \item We claim that only $\pi^*H$ does not suffice to generate the free part of $\Pic(X)$. On $\mathbf{P}^2$ with a general $8$-line arrangement, one can always find a line $l_0\subset\mathbf{P}^2$, such that it intersects with the branch divisor at distinct points $p_1,\ldots,p_4$ but $p_5=p_6$, $p_7=p_8$. Denote by $C_0$ the $G$-cover of $l_0$ and $C_0^\nu$ the normalization of $C_0$. According to \cite[Lemma~2.9(3)]{alexeev2024explicit}, in the building data of the cover $C_0^\nu\to l_0$, the branch data consist of only $p_1,\ldots,p_4$. When label these four points by $g\in G$, one can use $(1,*,*,0)$, thus these four elements do not generate the whole group $G$. So $C_0^\nu$ is not connected by \cite[Lemma~2.8]{alexeev2024explicit}, thus $C_0$ must be reducible.
\end{enumerate}
\end{remark}

In general, the deformation space of the cover might be larger than the deformation space of the base with branch divisors, as exemplified by a hyperelliptic curve with $g\geqslant 3$. 
For our pairs, one can always move four lines to the standard position by using $\PGL(\mathbb{C},3)$, thus the deformation space of such pairs $\Def(\mathbf{P}^2,D)$ has dimension $(8-4)\times2=8$. One may also consider another family of surfaces, which are double covers $\mathbf{P}^2$ branched along eight lines as in \cref{Rmk:surfaces_in_the_middle}\cref{Item:4_rmk_surfaces_in_middle}. The data in \cref{Sec:KSBA} can also be used for stable degenerations of this family. However, the eight lines could be smoothed to a planar octic curve, whose moduli is $\binom{8+2}{2}-\dim\PGL(\mathbb{C},3)-1=36$-dimensional. Now we compute the dimension of the deformation space of Persson surfaces, which is also $8$. So the compactified moduli in \cref{Thm:Persson_KSBA_in_intro} is complete, i.e., it parametrizes all possible $1$-parameter degenerations.

\begin{proposition}\label{Prop:deformation_dim}
    The deformation space of a smooth Persson surface $X$ has dimension $h^1(X,\mathcal{T}_X)=8$, where $\mathcal{T}_X$ denotes the tangent sheaf. 
\end{proposition}

\begin{proof}
    There are two short exact sequences associated with the generalized prolongation bundle $\mathcal{P}$ from \cite[Section~2]{pardini1998period}:
\begin{align}
    0&\rightarrow\Omega_{\mathbf{P}^2}^1\rightarrow\mathcal{P}\rightarrow\mathcal{O}_{\mathbf{P}^2}^{\oplus 8}\rightarrow 0,\label{Eqn:2.1}\\
    0&\rightarrow\bigoplus_{i=1}^8\mathcal{O}_{\mathbf{P}^2}(-D_i)\rightarrow\mathcal{P}\rightarrow\Omega_{\mathbf{P}^2}^1(\log D)\rightarrow 0\label{Eqn:2.2}.
\end{align}
    Taking the long exact sequence for the dual of \cref{Eqn:2.1}, one knows $h^0(\mathbf{P}^2, \mathcal{P}^*)=16$. From the long exact sequence for \cref{Eqn:2.2}, we obtain $h^1(\mathbf{P}^2,\mathcal{T}_{\mathbf{P}^2}(-\log D))=8$. Similarly, we also have  $h^1(\mathbf{P}^2,\mathcal{T}_{\mathbf{P}^2}(-\log D)\otimes\mathcal{O}_{\mathbf{P}^2}(-k))=0$, for $k=2,4$. Thus, according to the formula 
\[
    H^1(X,\mathcal{T}_X)\cong \bigoplus_{\chi\in G^*} H^1(\mathbf{P}^2,\mathcal{T}_{\mathbf{P}^2}(-\log D)\otimes L_{\chi}^{-1})
\]
    by Pardini in \cite[Proposition~4.1]{pardini1991abelian}, we get $h^1(X,\mathcal{T}_X)=8$.
\end{proof}

\begin{remark}
    According to \cite[Page~484]{catanese1984moduli}, for smooth surfaces with numerical invariants $\chi(\mathcal{O}_X)=4,K_X^2=16$, the expected minimal dimension of the moduli is  $10\chi(\mathcal{O}_X)-2K_X^2=10\cdot 4-2\cdot 16=8$. We do not know whether there exist other subfamilies of smooth surfaces with the same numerical invariants but larger moduli dimensions. The same computation as in the proof of \cref{Prop:deformation_dim} shows that the moduli of the smooth Campedelli surfaces in \cite{alexeev2024explicit} is $6$-dimensional, which is also the expected minimal dimension.
\end{remark}

One can also see the relation between the Hodge decomposition and eigenspaces.

\begin{proposition}\label{prop:eigen}
    Let $X\to Y=\mathbf{P}^2$ be an abelian cover for any fixed group $G=(\mathbb{Z}/2\mathbb{Z})^m$, with $X$ smooth and the branch divisor $D$ simply normal crossed. Let $S_\chi\to\mathbf{P}^2$  be the double cover associated to $L_{\chi}$, $0 \neq \chi\in G^*$, branched along $D_{\chi,\chi^{-1}}:=\sum_{g\in G\mid\chi(g)=-1}D_g$. Then the eigenspace $H^{p,q}_{\mathrm{prim}}(X)^\chi$ is naturally isomorphic to $H^{p,q}_{\mathrm{prim}}(S_\chi)$.
\end{proposition}
    Here by slightly abuse of notation, we still use $H^{p,q}_{\prim}(S_\chi)$ to denote the eigenspace $H^{p,q}(S_\chi)^{\chi}$ even if $S_\chi$ has $A_1$-singularities; note also that the Hodge structure on $H^2(S_\chi,\mathbb{Z})$ is pure.

\begin{proof}
    The non-primitive part corresponds to the character $\chi=0$. For $\chi\neq 0$, since $X$ is assumed to be smooth, by the formula in \cite[Proposition 1.2]{pardini1998period} we have:
\[
            H^{p,q}(X)^\chi\cong H^q(\mathbf{P}^2,\Omega_{\mathbf{P}^2}^p(\log D_{\chi,\chi^{-1}})\otimes L_{\chi}^{-1}).
\]
    While for the double cover $S_\chi$, we also have
\[
    H^{p,q}_{\mathrm{prim}}(S_\chi,\mathbb{C})=H^{p,q}(S_\chi,\mathbb{C})^{\chi}\cong H^q(\mathbf{P}^2,\Omega^p_{\mathbf{P}^2}(\log D_{\chi,\chi^{-1}})\otimes L_{\chi}^{-1}),
\]
    where the superscript $\chi$ denotes the eigenspace for the involution of $S_\chi$ and the last isomorphism is given in Arapura \cite[Lemma~1.2]{arapura2014}.
\end{proof}

\begin{corollary}\label{Cor:eigen_Persson}
    When $X$ is a smooth Persson surface, we have:
    \begin{enumerate}[label=\textup{(\arabic*)}]
        \item $h^{1,1}(X)^{\chi}=9$ if $\chi=\chi_0$ and $h^{1,1}(X)^{\chi}=1$ otherwise.
        \item $h^{2,0}(X)^{\chi}=3$ if $\chi=\chi_0$ and $h^{2,0}(X)^{\chi}=0$ otherwise.
    \end{enumerate}
\end{corollary}
\begin{proof}
    By \cref{Rmk:surfaces_in_the_middle}\cref{Item:4_rmk_surfaces_in_middle}, before resolving the $28$ $A_1$-singularities, $h^{1,1}(S_{\chi_0})=38-28=10$, so $h^{1,1}(X)^{\chi_0}=h^{1,1}_{\mathrm{prim}}(S_{\chi_0})=10-1=9$ according to \cref{prop:eigen}. The dimensions of other $G^*$-eigenspaces can be computed immediately in the same way. Notice that another way to obtain the results is by the short exact sequences \cref{Eqn:2.1}, \cref{Eqn:2.2} of the generalized prolongation bundle $\mathcal{P}$ without using \cref{prop:eigen}.
\end{proof}

\section{Compactifications of the moduli of Persson surfaces}
\subsection{KSBA moduli and stable degenerations}\label{Sec:KSBA}
We begin this section by reviewing the construction of KSBA compactified moduli and stable degenerations of weighted hyperplane arrangements following Alexeev \cite{alexeev2008weighted} and \cite[Chapter~5]{alexeev2015moduli}. In the special case when the weight is $\mathbf{1}=(1,\ldots,1)$, the construction also appears in the work of Hacking--Keel--Tevelev \cite{hacking2006compactification}, where stable degenerations are described in terms of visible contours in the sense of Kapranov \cite[Section~3.1]{Kap93Chow}. 

We refer to \cite[Definition~2.8]{kollar2013singularities} for the definitions of standard singularities of the minimal model program, in particular, the log canonical singularities. We now recall the definition of stable objects in the KSBA moduli theory.
\begin{definition}[{\cite[Definition-Lemma~5.10]{kollar2013singularities}}]
    Let $(X,D)$ be a pair, i.e., $X$ is a variety and $D$ is an effective $\mathbb{Q}$-divisor. $(X,D)$ is \emph{KSBA stable} or just \emph{stable} if:
    \begin{enumerate}
        \item $K_X+D$ is ample.
        \item The pair $(X,D)$ has semi-log canonical singularities, which means:
        \begin{enumerate}[label=(2\alph*)]
            \item $X$ is demi-normal, i.e., $X$ is either regular or nodal in codimension $1$ and satisfies Serre condition $S_2$.
            \item The support of the conductor $D_{\mathrm{cd}}$ does not contain any irreducible component of $D$.
            \item $K_X+D$ is $\mathbb{Q}$-Cartier.
            \item The pair $(X^{\nu},D_{\mathrm{cd}}^\nu+\nu^{-1}_*D)$ is log canonical, where $\nu\colon X^{\nu}\to X$ is the normalization map of $X$.
        \end{enumerate}
    \end{enumerate}
\end{definition}

According to \cite[Proposition~2.5]{alexeev2012non}, the problem of finding all stable degenerations of Persson surfaces reduces to the study of stable degenerations of the pair $(\mathbf{P}^2, \mathbf{b}D=\frac{1}{2}\sum_{i=1}^8 D_i)$. We can therefore apply the techniques developed by Alexeev \cite{alexeev2015moduli} for compactifying moduli spaces of weighted hyperplane arrangements. A summary of the theory is provided below, with the relevant definitions introduced later.

\begin{enumerate}[wide]
    \item For the moduli space, the compactification is given by the Chow quotient of the $\mathbf{b}$-weighted Grassmannian by the torus action.
    \item For the family, the $\mathbf{b}$-weighted stable hyperplane arrangements (a special case of KSBA stable pairs) are given by GIT quotients of the total space of the universal bundle on the Grassmannian $\Gr(d,n)$ restricted to embedded stable toric varieties in $\Gr(d,n)$, whose linearizations are determined by the weight $\mathbf{b}$.
\end{enumerate}

\begin{definition}\label{Def:matroidpoly}
    Fix two positive integers $d,n$ with $d\leqslant n$, the \emph{hypersimplex} is defined as
\[
    \Delta(d,n)=\{(x_1,\ldots,x_n)\in\mathbb{R}^n|0\leqslant x_i\leqslant 1\text{ for all }i,\sum_{i=1}^n x_i=d\}.
\]
    Let $e_i$ be the $i$-th unit vector in $\mathbb{R}^n$.
    A subpolytope $Q\subset\Delta(d,n)$ is a (\emph{$\mathbb{C}$-realizable}) \emph{matroid polytope} if there is a collection of nonzero $\mathbb{C}$-vectors $\{v_1,\ldots,v_n\}\subset\mathbb{C}^d$ such that $Q$ is the convex hull of $e_I=\sum_{i\in I}e_i$, where $|I|=d$ and $\{v_i\mid i\in I\}$ is a basis of $\mathbb{C}^d$.  By a \emph{complete matroid polytope tiling} or just \emph{matroid tiling} of $\Delta(d,n)$, we mean a decomposition $\Delta(d,n)=\cup Q_s$, where each $Q_s$ is a matroid polytope and the decomposition satisfies \emph{the face-fitting condition}, i.e., $Q_s\cap Q_{s'}$ is either empty or a proper face of both $Q_s$ and $Q_{s'}$. A \emph{partial matroid tiling} of $\Delta(d,n)$ is a face-fitting union $\cup Q_s\subset\Delta(d,n)$ by matroid polytopes which does not necessarily cover the whole $\Delta(d,n)$.
\end{definition}
\begin{remark}
    In this paper we only consider $\mathbb{C}$-realizable matroids. This is a special case of abstract matroids, whose polytopes are characterized in \cite[Theorem~4.1]{gelfand1987combinatorial} by vertices and edges. Moreover, we always require that the matroid polytopes $Q_s$ in the tiling are full dimensional, i.e., $\dim(Q_s)=n-1$ for all $s$.
\end{remark}

\begin{definition}[{\cite[Definition~4.4.1]{alexeev2015moduli}}]\label{Def:b-cut}
    For a weight $\mathbf{b}=(b_1,\ldots,b_n)\in\mathbb{Q}^n\cap(0,1]^n$, the \emph{$\mathbf{b}$-cut} of $\Delta(d,n)$ is
\[
    \Delta_{\mathbf{b}}(d,n)=\{(x_1,\ldots,x_n)\in\mathbb{R}^n,0\leqslant x_i\leqslant b_i\text{ for all }i\text{ and }\sum_{i=1}^n x_i=d\}.
\]
A \emph{matroid tiling} of $\Delta_{\mathbf{b}}(d,n)$ is a partial matroid tiling $\Delta(d,n)\supset \cup Q_s$ such that  $\Delta_{\mathbf{b}}(d,n)$ is contained in $\cup Q_s$ and  $Q_s\cap\Delta^\circ_{\mathbf{b}}(d,n)\neq\emptyset$ for every $Q_s$, where $\Delta^\circ_{\mathbf{b}}(d,n)$ denotes the relative interior of $\Delta_{\mathbf{b}}(d,n)$.
\end{definition}

\begin{definition}[{\cite[Definition~0.1.7]{Kap93Chow}}]\label{Def:ChowQ}
    Let $G$ be an algebraic group acting on a projective variety $X\subset\mathbf{P}^N$, then by generic flatness, there exists an open subset $U\subset X$ such that every orbit closure $\overline{G\cdot x},x\in U$ has the same dimension and degree, i.e., $U/G$ is embedded in the Chow variety parametrizing cycles on $X$. The \emph{Chow quotient $X\ChQ G$} is defined as the closure of $U/G$ in the Chow variety.
\end{definition}

There is a natural action on $\Gr(d,n)$ by the torus $T=\mathbb{G}_m^n/\diag\mathbb{G}_m$. The points in the Chow quotient $\Gr(d,n)\ChQ T$, with respect to the Pl\"ucker embedding $\Gr(d,n)\subset\mathbf{P}^N$, $N={\binom{n}{d}-1}$, parameterize the cycles on $\Gr(d,n)$ satisfying the following definition.
\begin{definition}[{\cite[Definition~1.1.5]{alexeev2002complete}}]\label{Def:STV}
    A projective variety $X$ is called \emph{stable toric} if the following holds.
\begin{enumerate}
    \item There is a $T$-action on $X$.
    \item Every irreducible component of $X$ is a projective toric variety.
    \item $X$ is semi-normal.
\end{enumerate}
\end{definition}

\begin{example}[{\cite[Section~2.2]{alexeev2015moduli}}]
    A tiling of a fixed lattice polytope defines a stable toric variety when the gluing data are specified.
\end{example}

When $\mathbf{b}=\mathbf{1}=(1,\ldots,1)$, generic $T$-orbit closures in $\Gr(d,n)$ are polarized toric varieties associated with $\Delta(d,n)$. And degenerations over $\Gr(d,n)\ChQ T$ are stable toric varieties in $\Gr(d,n)$ associated with regular (see \cref{Rmk:regularity_main_cpt_connected_cpt}) matroid tilings of $\Delta(d,n)$ according to \cite[Proposition~1.2.11,~1.2.15]{Kap93Chow}.

On fibers, the construction via GIT works for both the ordinary Grassmannian $\Gr(d,n)$ (in \cite{alexeev2015moduli}) and the weighted Grassmannian $\Gr_{\mathbf{b}}(d,n)$ (in \cite{alexeev2008weighted}, note that this is different from Corti--Reid's weighted Grassmannians in \cite{corti2002weighted}). The difference is that with weighted Grassmannians one can obtain a flat family of embedded stable toric varieties in $\Gr_{\mathbf{b}}(d,n)$, which is not true for the ordinary $\Gr(d,n)$ and a general weight $\mathbf{b}$.  In our setting, we work with ordinary $\Gr(d,n)$, since it is not a problem to lose the flatness when concerning fibers independently.

\begin{remark}
    Up to normalization, $\oM_{\mathbf{b}}(\mathbf{P}^{d-1},n)$ is isomorphic to the Chow quotient $\Gr_{\mathbf{b}}(d,n)\ChQ T$; see \cite[Section~4]{alexeevbrion2006stable} and \cite[Section~7]{alexeev2008weighted}.
\end{remark}

Start from a $T$-equivariantly embedded stable toric variety $Y=\cup Y_s\hookrightarrow\Gr(d,n)$. Remember that for a polarized toric variety, the image of the $T$-moment map is just the corresponding polytope. $Y$ is of \emph{type $\mathbf{b}$} if its $T$-moment map image $\cup Q_s$ is a matroid tiling of $\Delta_{\mathbf{b}}(d,n)$. The GIT construction including the $T$-linearization data are as follows. The $\mathbb{Q}$-line bundle is $L_{\mathbf{b}}=\mathcal{O}_{\mathbf{P}^N}(1)\boxtimes\mathcal{O}_{\mathbf{P}^{n-1}}(|\mathbf{b}|-d)$, where $|\mathbf{b}|=\sum b_i$. Fix the $T$-action on sections $t\cdot z_i^{\alpha_i}=t_i^{\alpha_i-b_i/|\mathbf{b}|}z_i^{\alpha_i}$, $\sum \alpha_i=d-|\mathbf{b}|$ on the second factor and the standard democratic linearization on the first one.  Let $\mathbb{L}$ be the universal bundle over $\Gr(d,n)$, and let $\mathbb{L}_Y$ be the restriction of $\mathbb{L}$ on $Y$. Notice that $\mathbb{L}_Y\subset\mathbf{P}^N\times\mathbf{P}^{n-1}$, so there is a GIT quotient of $\mathbb{L}_Y$ with respect to the aforementioned linearization. We obtain a pair
\[
    (\mathbb{L}_Y\sslash_{\mathbf{b}}T,\sum_{i=1}^n b_iB_i),
\]
where $B_i$ is the GIT quotient $(\mathbb{L}_Y\cap (\mathbf{P}^N\times H_i))\sslash_{\mathbf{b}}T$, and $H_i=V(z_i)\subset\mathbf{P}^{n-1}_z$ is the $i$-th coordinate hyperplane.

\begin{theorem}[{\cite[Theorem~5.3.2, 5.4.2]{alexeev2015moduli} and \cite[Lemma~7.7]{alexeev2008weighted}}]
    The pair $(\mathbb{L}_Y\sslash_{\mathbf{b}}T,\sum b_i B_i)$  is KSBA stable when $|\mathbf{b}|>d$. And the construction above gives all stable degenerations of $(\mathbf{P}^{d-1},\mathbf{b}D)$ when $Y$ runs over all embedded stable toric varieties in $\Gr(d,n)$ of type $\mathbf{b}$, which are parametrized by a fine projective moduli space, 
    whose main irreducible component is $\overline{\mathbf{M}}_{\mathbf{b}}(\mathbf{P}^{d-1},n)$.
\end{theorem}

\begin{remark}\label{Rmk:regularity_main_cpt_connected_cpt}
    In general, for hyperplane arrangements, there exist non-smoothable degenerations, which are parametized by extra irreducible components of the KSBA moduli. For instance, when the weight $\mathbf{b}=\mathbf{1}$, all stable degenerated hyperplane arrangements are parametrized by the \emph{limit quotient}, i.e., the inverse limit of the variational GIT quotients $\Gr(d,n)\sslash_L T$ for all $T$-linearized ample line bundles, which is often reducible. Up to normalization, the main component of the limit quotient is isomorphic to $\Gr(d,n)\ChQ T$ by \cite[Corollary~3.6]{BakerHausenSimon12}, which is also a bijective homeomorphism by \cite[Theorem~3.8]{hu2005topological}. In terms of the language of polytopes, smoothable degenerations are from \emph{regular} (also called \emph{coherent} or \emph{convex}) matroid tilings, as defined in \cite[Definition~\rom{7}.1.3]{Gelfand1994},  of $\Delta(d,n)$. It is known that all matroid tilings of $\Delta(3,n)$ are regular for $n\leqslant 8$, so all degenerations in our situation are smoothable. 
    When $n\geqslant9$, the whole KSBA moduli of stable hyperplane arrangements is no longer irreducible; for example, see \cite[Section~7]{hacking2006compactification} for the case $n=9$ with combinatorial and geometric interpretations.
\end{remark}

Now we specialize to the weighted line arrangement problem for Persson surfaces, i.e., $d=3,n=8$ and the weight $\mathbf{b}=(\frac12,\ldots,\frac12)$.

\begin{lemma}\label{Lem:cut_intersection}
    Denote $x_I=\sum_{i\in I}x_i$ for $I\subset\{1,\ldots,n\}$. Up to symmetry, the only non-trivial matroid tilings of $\Delta_{\mathbf{b}}(3,8)$, are given as follows.
    \begin{enumerate}[label=\textup{(\arabic*)}]
        \item $\Delta_{\mathbf{b}}(3,8) \subset Q_1\cup Q_2\subset\Delta(3,8)$, where 
        \begin{align*}
            Q_1=\{x_{123}\leqslant 1\}\cap\Delta(3,8),\\
            Q_2=\{x_{45678}\leqslant 2\}\cap\Delta(3,8).
        \end{align*}
        \item $\Delta_{\mathbf{b}}(3,8) \subset Q_1\cup Q_2\cup Q_3\subset\Delta(3,8)$, where 
        \begin{align*}
            Q_1&=\{x_{123}\leqslant 1\}\cap\Delta(3,8),\\
            Q_2&=\{x_{12345}\leqslant 2,x_{45678}\leqslant 2\}\cap\Delta(3,8),\\
            Q_3&=\{x_{678}\leqslant 1\}\cap\Delta(3,8).
        \end{align*}
    \end{enumerate}
\end{lemma}
\begin{proof}
   This is given by two observations: first, through $\Delta^\circ_{\mathbf{b}}(3,8)$, the only hyperplane defined by the rank condition is $x_I=1,|I|=3$ (or equivalently, $x_I=2,|I|=5$). Second, all the facets (codimension $1$ faces) of a matroid polytope can be given by \emph{non-degenerate flats} (see \cite[Definition~4.2.1~and~Theorem~4.2.2]{alexeev2015moduli}), and the inequalities given by these flats have to be compatible to give a set of $\mathbb{C}$-vectors. Moreover, for a matroid polytope, the facet defined by $x_{I}\leqslant 1$, $|I|=3$ appears at most once. Otherwise, either $x_{I_1}\leqslant 1$, $x_{I_2}\leqslant 1$ with $I_1\cap I_2\neq\emptyset$ gives $I_1,I_2$ are not flats; or $I_1\cap I_2=\emptyset$ forces $x_{j_1}=x_{j_2}=\frac12$, $j_1,j_2\not\in I_1\cup I_2$, then the intersection of this matroid polytope with $\Delta_{\mathbf{b}}(3,8)$ is not full dimensional. Additionally, $x_I\leqslant 3, |I|=3$ and $x_J\leqslant 5, |J|=5$ cannot appear simultaneously if $I\subset J$. Otherwise, the intersection of the hyperplane $x_I=3$ and $x_J=5$ forces $x_{j_1}=x_{j_2}=\frac12$, $j_1,j_2\in J\setminus I$, on $\Delta_{\mathbf{b}}(3,8)$. Notice that for $Q_2$ in the second case, there is also an inequality $x_{45}\leqslant 1$ for defining the matroid polytope (so-called the \emph{parent} of $Q_2$ in \cite[Definition~2.3]{alexeev2008weighted}),  which we do not use, because this facet of the matroid polytope turns out to be a small face (not a facet) of  $Q_2$ after its parent intersecting with $\Delta_{\mathbf{b}}(3,8)$.
\end{proof}

\begin{remark}
    For the matroid tilings in \cref{Lem:cut_intersection}, we only consider the facets of $Q_s$ lying on the hyperplane passing through $\Delta_{\mathbf{b}}^\circ(3,8)$.
    There exist other matroid tilings of $\Delta_{\mathbf{b}}(3,8)$ with facets of $Q_s$ away from $\Delta_{\mathbf{b}}(3,8)$. Such partial tilings do not produce new types of degenerations, because facets of $Q_s$ away from $\Delta_{\mathbf{b}}(d,n)$ correspond to log canonical singularities. For example, if one adds $x_{456}\leqslant 2$ on $Q_1$ in the first case, this just means that the three lines $D_4,D_5,D_6$ pass through a common point, which is log canonical according to \cref{Lem:lc_arrangements}. Such partial matroid tilings correspond to non-generic stable degenerations in each type. Additionally, we point out that every partial matroid tiling of $\Delta_{\mathbf{b}}(3,8)$ extends to a complete tiling of $\Delta(3,8)$; see \cite[Section~7]{shin2023} for a discussion of when the extension fails for larger $n$, as predicted in \cite[Question~2.10]{alexeev2008weighted}.
\end{remark}

\begin{lemma}[{\cite[Theorem 4.2.2]{alexeev2015moduli}}]\label{Lem:lc_arrangements} For any weight $\mathbf{b}\in(0,1]^n\cap\mathbb{Q}^n$, a $\mathbf{b}$-weighted line arrangement $(\mathbf{P}^2,\mathbf{b}D)$ is log canonical if and only if:
\begin{enumerate}
    \item The sum of weights for lines coincided is at most $1$.
    \item The sum of weights for lines through a common point is at most $2$.
\end{enumerate}
\end{lemma}

When Persson surfaces degenerate, on the base of the $G$-cover, the stable degenerations of the pair $(\mathbf{P}^2,\mathbf{b}D)$ can be read off immediately from the polytope tiling data.
\begin{theorem}\label{Thm:Persson_KSBA}
    Stable degenerations of Persson surfaces are $G$-covers of degenerations of $\mathbf{P}^2$ of the following three types, as shown in \cref{Fig:degenerationofP2}, with the degenerated line arrangements:
\begin{enumerate}[label=Type \Roman*:, align=left]
    \item[Type $0$:] $Y\cong\mathbf{P}^2$.
    \item $Y=Y_1\cup Y_2$, where $Y_1\cong \mathbf{P}^2$ and $Y_2\cong\mathbf{F}_1$.
    \item $Y=Y_1\cup Y_2\cup Y_3$, where $Y_1\cong Y_3\cong\mathbf{P}^2$ and $Y_2\cong \mathbf{P}^1\times\mathbf{P}^1$.
\end{enumerate}
\end{theorem}
\begin{proof}
      According to \cref{Lem:cut_intersection}, the degenerations of $\mathbf{P}^2$ are given by rank conditions of the facets of matroid polytopes. We thus obtain the three types of degenerations as given in \cref{Fig:degenerationofP2}. For example, for the matroid polytope $Q_2$ in type \rom{2}, the inequality $x_{12345}\leqslant 2$ corresponds to the degeneration with $D_1,\ldots,D_5$ through a common point $p$; and similarly, $D_4,\ldots,D_8$ pass through a common point $q$ from $x_{45678}\leqslant 2$ (see \cref{Fig:the_unstable_q_2_in_type_2}). Blow up these two points on a $1$-parameter family, then contract the strict transform of the line connecting $p,q$. One obtains the irreducible component $Y_2\cong\mathbf{P}^1\times\mathbf{P}^1$, which corresponds to the rhombus in the middle of \cref{Fig:degenerationofP2}, Type \rom{2}. \qedhere
      
\begin{center}
\begin{figure}[htpb]
\begin{tikzpicture}[line width=1pt, color=SkyBlue, scale=1.1]
    \draw (-3,0.02) -- (3,0.02);
    \draw (-3,-0.02) -- (3,-0.02);
    \draw (-2.5,-1) -- (-0.5,1);
    \draw (-2.5,1) -- (-0.5,-1);
    \draw (-1.5,-1) -- (-1.5,1);
    \draw (2.5,-1) -- (0.5,1);
    \draw (2.5,1) -- (0.5,-1);
    \draw (1.5,-1) -- (1.5,1);
    \node [circle, fill=DarkRed, scale=0.5] at (-1.5,0) {};
    \node [circle, fill=DarkRed, scale=0.5] at (1.5,0) {};
    \node at (-2.5,-1.4) [color=black] {$D_1$};
    \node at (-1.5,-1.4) [color=black] {$D_2$};
    \node at (-0.5,-1.4) [color=black] {$D_3$};
    \node at (0.5,-1.4) [color=black] {$D_6$};
    \node at (1.5,-1.4) [color=black] {$D_7$};
    \node at (2.5,-1.4) [color=black] {$D_8$};
    \node at (4,0) [color=black] {$D_4=D_5$};
\end{tikzpicture}
\caption{An unstable degeneration corresponds to $Q_2$ in  type \rom{2}.}
\label{Fig:the_unstable_q_2_in_type_2}
\end{figure}
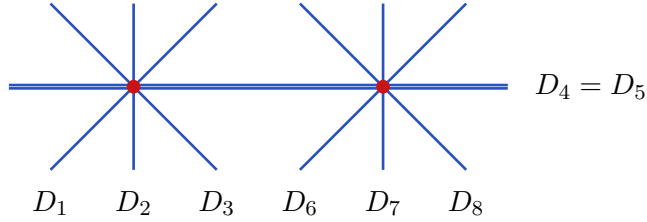
\end{center}
\end{proof}

\begin{theorem}\label{Thm:smoothness_stack}
    The KSBA compactification $\Ms$ of the moduli space of Persson surfaces is
\[
    \oM_{\mathbf{b}}(\mathbf{P}^2,8)/G_{\Stab},
\]
    where $\mathbf{b}=(\frac12,\ldots,\frac12)$ and  $G_{\Stab}\cong\mathrm{Aff}(\mathbb{F}_2,3)=(\mathbb{Z}/2\mathbb{Z})^3\rtimes\GL(\mathbb{F}_2,3)$. $\Ms$ is unirational. Moreover, the KSBA compactified moduli stack $\overline{\mathcal{M}}^{\mathrm{s}}$ of Persson surfaces is smooth.
\end{theorem}
\begin{proof}
    $G_{\Stab}$ permutes labels of branch divisors but fixes the line bundles, equivalently, $\chi_0$ and the set $\{g\in G\mid D_g\neq\emptyset\}$ are fixed under $G_{\Stab}$ according to \cref{Rmk:ab_cover_uniqueness}. So $G_{\Stab}$ is a subgroup of $\GL(\mathbb{F}_2,4)$, also of $\mathfrak{S}_8$, of order $\prod_{i=1}^3(2^4-2^i)=1344$. It is well known that the only subgroup of $\mathfrak{S}_8$ with this order is isomorphic to the group of affine transformations of the vector space $\mathbb{F}_2^3$. The unirationality follows from the rationality of $\oM_{\mathbf{b}}(\mathbf{P}^2,8)$, which is given by wall crossing birational equivalence between KSBA compactifications and the rationality of the generalized Losev--Manin moduli with weights $\mathbf{b}_{\mathrm{LM}}:=(1,1,1,\epsilon,\ldots,\epsilon)$ for $\epsilon\in\mathbb{Q}$ and $0<\epsilon\ll 1$. Actually, $\oM_{\mathbf{b}_{\mathrm{LM}}}(\mathbf{P}^2,n)$ is a projective toric variety; see \cite[Example~9.6]{alexeev2008weighted}.

    For the smoothness of the stack, it suffices to prove the smoothness of the fine moduli space $\oM_{\mathbf{b}}(\mathbf{P}^2,8)$, which will be implied by the unobstructedness of all stable pairs. In the cases of type 0, \rom{1}, notice that $D$ is Cartier and $H^1(\mathcal{O}_Y(D))=0$ by \cref{Lem:Cartier_and_H1=0}. So by the argument of \cite[Theorem~3.12]{hacking2004compact}, one can ignore the divisorial part, i.e., the expected unobstructedness of the pairs is implied by the unobstructedness only of the variety parts. Now we need $\mathbf{T}^2_Y=0$ (see \cref{Def:three_sheaves_spaces}), which is given in \cref{Lem:T2} together with the local-global spectral sequence $E^{p,q}_2:=H^p(\mathcal{T}_Y^q)\Rightarrow \mathbf{T}^{p+q}_Y$. In particular, the unobstructedness of type \rom{2} degenerations is given in \cref{Lem:type_2_obstruction}.
\end{proof}

\begin{remark}\label{Rmk:stack_band}
    In \cref{Thm:smoothness_stack}, we proved the smoothness of the stack of $G$-covers of the KSBA stable $\mathbf{b}$-weighted line arrangements. We do not know whether the compactified moduli stack of Persson surfaces is a connected component among the moduli stack of surfaces with the same numerical invariants. The same question remains open even for Campedelli surfaces (see \cite[Remark~1.1]{alexeev2024explicit}). To address this problem, one needs to verify the $\mathbb{Q}$-Gorenstein unobstructedness for all possible stable degenerations of the surfaces.
\end{remark}

After presenting some needed results on KSBA wall crossings and deformation theory, we prove the lemmas used in the proof of \cref{Thm:smoothness_stack}.

\begin{definition}[{\cite[Definition~5.5.1]{alexeev2015moduli}}]\label{Def:weight_domain}
    The \emph{weight domain} of all possible weights for hyperplane arrangements is
\[
    \mathcal{D}(d,n)=\{\mathbf{b}=(b_1,\ldots,b_n)\in\mathbb{Q}^n|0<b_i\leqslant 1\text{ for all }i\text{ and }\sum_{i=1}^nb_i>d\}.
\]
A \emph{wall} is a hyperplane in $\mathbb{Q}^n$ defined by $\{x_I(=\sum_{i\in I}x_i)=k\}$ for some $I\subset[n]$ with $2\leqslant |I|\leqslant n-2$ and $1\leqslant k\leqslant d-1$.  $\mathcal{D}(d,n)$  is subdivided by walls into finitely many relatively open \emph{chambers}. Denote by $\mathrm{Ch}(\mathbf{b})$ the chamber containing the weight $\mathbf{b}$.
\end{definition}

\begin{theorem}[{\cite[Theorem~5.5.2]{alexeev2015moduli}}]
\label{Thm:Alexeev_wall_crossing}
    When the weight vector varies on the weight domain $\mathcal{D}(d,n)$, the corresponding moduli spaces of weighted stable hyperplane arrangements have the following properties.
\begin{enumerate}[label=\textup{(\arabic*)}]
    \item  If $\mathbf{b}$ and $\mathbf{b}^{\prime}$ are in the same chamber, then $\oM_{\mathbf{b}}(\mathbf{P}^{d-1},n)\cong \oM_{\mathbf{b^{\prime}}}(\mathbf{P}^{d-1},n)$.
    \item If $\mathbf{b}\in\overline{\mathrm{Ch}(\mathbf{b}^{\prime})}$, then there is a morphism $\oM_{\mathbf{b}^{\prime}}(\mathbf{P}^{d-1},n)\to \oM_{\mathbf{b}}(\mathbf{P}^{d-1},n)$, which lifts to families $(\mathcal{X}^{\prime},\mathbf{b}^{\prime}\mathcal{B}^{\prime})\to(\mathcal{X},\mathbf{b}\mathcal{B})$.
    \item If $\mathbf{b}\in\overline{\mathrm{Ch}(\mathbf{b}^{\prime})}$ and $\mathbf{b}^{\prime}\leqslant\mathbf{b}$, i.e., $b_i^{\prime}\leqslant b_i$ for all $i$, then $\oM_{\mathbf{b}}(\mathbf{P}^{d-1},n)\cong\oM_{\mathbf{b}^{\prime}}(\mathbf{P}^{d-1},n)$. \label{Item:3_Thm:Alexeev_wall_crossing}
\end{enumerate}
\end{theorem}
\begin{remark}
    For the last part in \cref{Thm:Alexeev_wall_crossing}, in general, for moduli problems other than hyperplane arrangements with general klt fibers, the isomorphism is given after taking normalization; see the discussion in Ascher--Bejleri--Inchiostro--Patakfalvi \cite[Section~8.1]{ascher2023wall}. Further generalizations can be found in Meng--Zhuang \cite{meng2023mmplocallystablefamilies}.
\end{remark}

\begin{definition}\label{Def:three_sheaves_spaces}
    Let $\mathbf{L}_X$ be the cotangent complex over an algebraic variety $X$. The \emph{tangent}, \emph{deformation}, \emph{obstruction sheaf/space of $X$}, are defined as follows for $i=0,1,2$ respectively:
    \begin{enumerate}
        \item $\T_{X}^i=\sExt^i(\mathbf{L}_{X/\mathbb{C}},\mathcal{O}_X)$.
        \item $\mathbf{T}_{X}^i=\Ext^i(\mathbf{L}_{X/\mathbb{C}},\mathcal{O}_X)$.
    \end{enumerate}
\end{definition}

\begin{remark}\label{Rmk:Q-Gorenstein_def}
To preserve the numerical invariants (see the examples in \cite[Section~1.4]{kollar2023families} for violators), the correct obstruction space for the KSBA moduli stack is the $\mathbb{Q}$-Gorenstein obstruction space, which is given by the canonical index $1$ covers. For type $0$, \rom{1} and \rom{2}$^\prime$ degenerations (see \cref{Fig:type_2_modification}), the $\mathbb{Q}$-Gorenstein versions of the sheaves/spaces in \cref{Def:three_sheaves_spaces} are as the same as the ordinary ones. We refer to \cite[Section~2.3]{kollar2013singularities}, \cite[Section~3]{hacking2004compact} and \cite[Section~5]{abramovich2011stable} for details of canonical index $1$ covers and $\mathbb{Q}$-Gorenstein deformations.
\end{remark}

\begin{center}
\begin{figure}[htpb]
\begin{tikzpicture}[scale=1.1, line width=1pt]
        \filldraw [gray, opacity=0.3] (0, 0) -- (3, 0) -- ({3*cos(60)}, {3*sin(60)}) -- cycle;
        \draw[SkyBlue, line width=1pt] (2, 0) -- (3, 0) -- ({3*cos(60)}, {3*sin(60)}) -- (0,0) -- (1,0);
        \draw[SkyBlue, line width=1pt] (1,0.04) -- (2,0.04);
        \draw[SkyBlue, line width=1pt] (1,0) -- (2,0);
        \coordinate (A) at (1.5,0);
        \coordinate (a) at ([shift={(-0.35,0.35*1.732)}] (1,0);
        \draw[SkyBlue, line width=1pt] (a) -- ++(60:2);
        \coordinate (b) at ([shift={(-0.2,0.2*1.732)}] (1,0);
        \draw[SkyBlue, line width=1pt] (b) -- ++(60:2);
        \coordinate (e) at ([shift={(0.35,0.35*1.732)}] (2,0);
        \draw[SkyBlue, line width=1pt] (e) -- ++(120:2);
        \coordinate (f) at ([shift={(0.2,0.2*1.732)}] (2,0);
        \draw[SkyBlue, line width=1pt] (f) -- ++(120:2);
        \draw[SkyBlue, line width=1pt] (a) -- ++(-120:0.6);
        \draw[SkyBlue, line width=1pt] (b) -- ++(-160:0.6);
        \draw[SkyBlue, line width=1pt] (e) -- ++(-60:0.6);
        \draw[SkyBlue, line width=1pt] (f) -- ++(-20:0.6);
        \draw[SkyBlue, line width=1pt] (1,0) -- ++(140:0.75);
        \draw[SkyBlue, line width=1pt] (2,0) -- ++(40:0.75);
        \draw[line width=2pt, white] (1,0) -- ++(120:1);
        \draw[line width=2pt, white] (2,0) -- ++(60:1);
        \node at (1.5,-0.5) {Type \rom{2}$'$};
\end{tikzpicture}
\caption{The KSBA stable replacements of type \rom{2} degenerations for weight $\mathbf{a}$. The broken line with weight $(\frac12-\varepsilon)$ is the boundary bottom broken line.}
\label{Fig:type_2_modification}
\end{figure}
\end{center}

\begin{lemma}\label{Lem:type_2_obstruction}
    On $\overline{\mathcal{M}}^s$, the points parametrizing type \rom{2} degenerations are smooth.
\end{lemma}
\begin{proof}
    In the case of type \rom{2} degenerations, notice that the restriction of $K_Y$, which is the boundary of the big broken triangle in \cref{Fig:degenerationofP2}, Type \rom{2}, on the irreducible component $Y_2$ is a single point, not divisorial. Then the ambient variety $Y$ is no longer $\mathbb{Q}$-Gorenstein, so Hacking's strategy does not apply directly. 
    We introduce another weight $\mathbf{a}=(\frac12,\ldots,\frac12,\frac12-\epsilon)\in\mathbb{Q}^8$, where $0<\epsilon\ll 1$. According to \cref{Thm:Alexeev_wall_crossing}, since $\mathbf{b}\in\overline{\mathrm{Ch}(\mathbf{a})}$, there is an isomorphism $\rho_{\mathbf{b},\mathbf{a}}\colon\oM_{\mathbf{a}}(\mathbf{P}^2,8)\xrightarrow{\sim}\oM_{\mathbf{b}}(\mathbf{P}^2,8)$. Now $K_Y+\mathbf{a}D$ is not $\mathbb{Q}$-Cartier as it intersects with the component $Y_2\cong\mathbf{P}^1\times\mathbf{P}^1$ is not codimension $1$. When lifted to families over $\rho_{\mathbf{b},\mathbf{a}}$, the preimages of type \rom{2} degenerations are given as in \cref{Fig:type_2_modification}, which is a small birational modification, denoted by type \rom{2}$'$. Now it suffices to prove the unobstructedness of type \rom{2}$'$ degenerations with respect to weight $\mathbf{a}$, on which Hacking's strategy applies due to \cref{Lem:Cartier_and_H1=0}. Again, the unobstructedness is obtained by \cref{Lem:T2}.
\end{proof}

\begin{lemma}\label{Lem:Cartier_and_H1=0}
    For every degeneration of type $0$, \rom{1} and \rom{2}$^\prime$, we have
    \begin{enumerate}[label=\textup{(\arabic*)}]
        \item The divisorial part $D$ is Cartier.
        \item $H^1(\mathcal{O}_Y(D))=0$.
    \end{enumerate} 
\end{lemma}
\begin{proof}
    \begin{enumerate} [wide]
        \item In all above cases, for any point $p$ on the double locus, the multiplicities of lines through $p$ on different irreducible components coincide. Thus $D$ is Cartier.
        \item By Serre duality, $H^1(\mathcal{O}_Y(D))=H^1(\mathcal{O}_Y(K_Y-D))^\vee$. Notice that $-K_Y+D$ is ample for all cases, so the vanishing is implied by \cite[Corollary~6.6]{kovacs2010canonical}.\qedhere
    \end{enumerate}
\end{proof}

\begin{lemma}\label{Lem:T2}
    Let $Y$ be a degeneration of $\mathbf{P}^2$ of type $0$, \rom{1} or \rom{2}$^\prime$, we have the following vanishings:
\[
    H^0(\mathcal{T}_Y^2)=H^1(\mathcal{T}_Y^1)=H^2(\mathcal{T}_Y^0)=0.
\]
\end{lemma}
\begin{proof}
    The result is clear for type $0$. Notice that all these types of degenerations only have double normal crossings, thus locally complete intersections (LCI). By \cite[Theorem~10.7]{hacking2001compactification}, which reads that $\mathcal{T}^2_Y$ could support only on non-LCI locus, we have $\mathcal{T}_Y^2=0$, so $H^0(\mathcal{T}_Y^2)=0$.

    For type \rom{1}, denote the double line by $C\cong\mathbf{P}^1$, by \cite[Proposition~3.6]{hassett1999stable}, $\mathcal{T}_Y^1=\mathcal{O}_{Y_1}(C|_{Y_1})|_C\otimes\mathcal{O}_{Y_2}(C|_{Y_2})|_C\cong\mathcal{O}_C(1)\otimes\mathcal{O}_C(-1)\cong\mathcal{O}_C$, which is finitely generated by global sections. So by \cite[Theorem~4.10]{tziolas2015smoothings}, we obtain $H^1(\mathcal{T}_Y^1)=H^2(\mathcal{T}_Y^0)=0$. The discussion for type \rom{2}$^\prime$ is the same since $\mathcal{T}_Y^1\cong\mathcal{O}_{C_1}\oplus\mathcal{O}_{C_2}$ for the two double lines $C_1,C_2\cong\mathbf{P}^1$.
\end{proof}

Up to this point, we have filled in the details of the proof of \cref{Thm:smoothness_stack}. Next, we describe the geometry of stable Persson surfaces that are generic members of type \rom{1}, \rom{2}.
\begin{proposition}\label{Prop:degeneration_cover}
    Let $X_s$ be the generic $G$-cover of the irreducible component $Y_s$, then we have:
\begin{enumerate}[label=\textup{(\arabic*)}]
    \item For generic type \rom{1} degenerations, the $G$-cover $X_1$ over $Y_1\cong\mathbf{P}^2$ is a K3 surface with eight $A_1$-singularities; and the $G$-cover $X_2$ over $Y_2\cong\mathbf{F}_1$ is an elliptic surface with numerical invariants $p_g(X_2)=2,q(X_2)=0$ and $h^{1,1}(X_2)=22$ and eight $A_1$-singularities. $X_1$, $X_2$ are glued along an elliptic curve.
     \item For generic type \rom{2} degenerations, the $G$-cover $X_1,X_3$ over $Y_1\cong Y_3\cong\mathbf{P}^2$ are K3 surfaces with four $A_3$-singularities; the $G$-cover $X_2$ over $Y_2\cong\mathbf{P}^1\times\mathbf{P}^1$ is a smooth K3 surface. $X_2$, $X_j$ are glued along an elliptic curve for $j=1,3$.
\end{enumerate}
\end{proposition}
\begin{proof}
    For generic degenerations of type \rom{1}, notice that on $Y_1=\mathbf{P}^2$, subindices of the five $D_g|_{Y_1}$'s have to generate $G$, so $X_1$ is connected by \cite[Lemma 2.7]{alexeev2024explicit}. In addition, $X_1$ is irreducible by \cite[Lemma~2.8(2)]{alexeev2024explicit} together with the generic assumption. The degenerate divisors in the building data are as follows. 
    There are three $D_g$'s on $Y_2=\mathbf{F}_1$ linear equivalent to $s_+$. Among the other five $D_{g_i},i=1,\ldots,5$, there is a unique subindex $g_j$ such that the sum of the other four $g_i$ is $0$. The four $D_{g_i},i\neq j$ are broken lines, i.e., the union of a line on $Y_1$ and a ruling on $Y_2$ glued along a point on the double locus. The double line $C\cong\mathbf{P}^1$ joins $D_{g_j}$; see the left picture of \cref{Fig:divisor_join}. Thus, one knows that $K_{X_1}\sim_{\mathbb{Q}}\pi^*(K_{\mathbf{P}^2}+\frac12\cdot 6H)=0$, $h^{2,0}(X_1)=1$, so $X_1$ is K3. The preimages of the intersection of the $D_{g_j}$ and $C$ give eight $A_1$-singularities. On the other component, $K_{X_2}\sim_{\mathbb{Q}}\pi^*(K_{\mathbf{F}_1}+\frac52 F+\frac32 s_+ +\frac12 s_-)=\pi^*F$, so $K_{X_2}^2=0$. The cover over a general ruling $F$ is branched along four points, which is an elliptic curve. The $A_1$-singularities are from preimages of $(D_{g_j}\cap C)|_{Y_2}$. Other numerical invariants can be similarly obtained as in \cref{Prop:numerical_data_of_Persson}.  
    Note that there are four intersection points on the double curve $\mathbf{P}^1$, and the four labels of the corresponding $D_g$ always generate the whole group $(\mathbb{Z}/2\mathbb{Z})^3$. Thus by \cite[Lemma~2.8]{alexeev2024explicit} and pull-back formula of canonical divisors, the double locus consists of an elliptic curve. Actually, the cover over the double locus $\mathbf{P}^1$ is an almost uniform cover with respect to the group $(\mathbb{Z}/2\mathbb{Z})^3$.

    For generic degenerations of type \rom{2}, connectedness and irreducibility of $X_i$ are given by the same reason as in type \rom{1}. There are three divisors whose restrictions on $Y_1\cong\mathbf{P}^2$ and $Y_2\cong \mathbf{P}^1\times\mathbf{P}^1$ are all divisorial, denoted by $D_{g_1},D_{g_2},D_{g_3}$. The other three lines with divisorial restrictions on $Y_2$ and $Y_3$ are labeled by $D_{g_6},D_{g_7},D_{g_8}$.
    Notice that on $Y_1$, the double line $C_1=Y_1\cap Y_2$ has to join the branch divisor, and locally the sum of labels of $C_1$ and $\sum_{i=1}^3 g_i$ must be $(0,0,0,0)$ since there exists a double cover as an irreducible component of the degeneration of the surface in \cref{Rmk:surfaces_in_the_middle}\cref{Item:4_rmk_surfaces_in_middle}. Say, $C_1$ joins $D_{g_4}$, then $C_2=Y_2\cap Y_3$ cannot also join $D_{g_4}$, because the sum of any other three $g_j$'s must be different from $g_4$; see the right picture of \cref{Fig:divisor_join} where lines in the same color correspond to the same $g_i$. So on $Y_2$, there are eight different labeled divisors, and $X_2$ is a smooth K3 surface. On $Y_1$ (and $Y_3$), there is a point through three lines, labeled as two $D_g$'s. According to \cite[Table~1,~3.2]{alexeev2012non}, $X_1$ (and $X_3$) is a K3 surface with four $A_3$-singularities.
\end{proof}  

\begin{center}
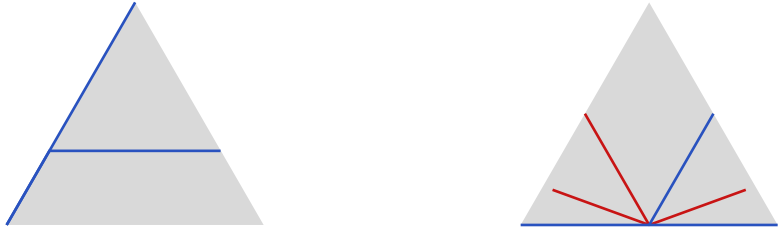
\begin{figure}[htpb]
\begin{tikzpicture}[line width=1pt, scale=1.7]
\begin{scope}
    \fill[gray, opacity=0.3] (-1,0) -- +(60:2) -- (1,0) -- cycle;
    \draw [SkyBlue] (-1,0) -- +(60:2);
    \draw [SkyBlue] (-1,0) -- ++(60:0.666) -- ++(1.332,0);
\end{scope}
\begin{scope}[xshift=4cm]
    \fill[gray, opacity=0.3] (-1,0) -- +(60:2) -- (1,0) -- cycle;
    \draw [DarkRed] (0,0) -- +(120:1);
    \draw [DarkRed] (0,0) -- +(160:0.8);
    \draw [DarkRed] (0,0) -- +(20:0.8);
    \draw [SkyBlue] (-1,0) -- (1,0);
    \draw [SkyBlue] (0,0) -- +(60:1);
\end{scope}
\end{tikzpicture}
\caption{Double lines join the branch divisors.}
\label{Fig:divisor_join}
\end{figure}
\end{center}

\begin{remark}
    In the case of type \rom{1}, on $Y_2$, each pair of the three divisors linear equivalent to $s_+$ intersect at a point. The curve which covers $F$ through such an intersection point is nodal. 
    So the component $X_2$ has three singular fibers. It will be interesting to compare this with Viehweg--Zuo \cite[Theorem 0.2]{ViehwegZuo2001}, which reads that any smooth variety with a surjective morphism to $\mathbf{P}^1$ and Kodaira dimension $\kappa(X)\geqslant 0$ has at least three singular fibers..
\end{remark}
    Now we list all possible singularities appearing on stable Persson surfaces.
\begin{proposition}
    In terms of the notation of the tables in Alexeev--Pardini \cite[Theorem~3.8,~3.9]{alexeev2012non}, all singularities appearing on stable degenerated Persson surfaces are as follows.
    \begin{enumerate}[label=\textup{(\arabic*)}]
        \item Type $0$: labeled by those with relations of length 4 or without relations, which are $3.1$, $4.1$, $4.4$, $2'.1$, $3'.1$, $4'.1$, $4'.7$, $4''.1$ and $4''.5$.
        
        \item Type \rom{1} over the double line: R$0.1$, R$2.1$, R$2.3$, R$4.8$, R$4.9$, R$4'.10$, R$4'11$, R$4''.8$ and R$4''.9$.
        \item Type \rom{2}, over the double line: R$0.1$, R$2.1$, R$4'.10$, R$4''.8$; and over the point connecting the three components: a gluing of two singularities, either type $3.2$ or $3'.3$, with a smooth surface.
    \end{enumerate}
\end{proposition}
\begin{proof}
 This is done by checking all possible choices of $g_i$'s and their relations. Recall that by \cref{Lem:lc_arrangements}, on each irreducible component we have at most two lines overlapped and at most four lines through a common point, since the degenerations have semi-log canonical singularities. In the case of type $0$ degenerations, for any three lines $D_{g_i},i=1,2,3$, $g_1+g_2+g_3=0$ never hold due to the condition of the almost uniform cover construction $\chi_0(g_i)=-1,i=1,2,3$, so we do not have length $3$ relations. The length $4$ relation $\sum_{i=1}^4 g_i=0$ can be given, for instance, by choosing $g_1=(1,1,1,0)$, $g_2=(1,1,0,1)$, $g_3=(1,0,1,1)$ and $g_4=(1,0,0,0)$. And the ``no relation'' case could be $g_1=(1,1,0,0)$, $g_2=(1,0,1,0)$, $g_3=(1,0,0,1)$ and $g_4=(1,0,0,0)$.

    For type \rom{1} and \rom{2} degenerations, the arguments are similar. The only subtlety concerns the double line joins the branch divisor as explained in the proof of \cref{Prop:degeneration_cover}. For example, in the case of type \rom{1} degenerations, if the reducible $D_{g_j}$, which consists of one original broken line and the double locus, is overlapped with  a broken line $D_{g_i}$ on one irreducible component but not the other, then the singularity is type R$4'11$ in \cite[Table~8]{alexeev2012non}.
\end{proof}

\begin{remark}
     There are two types of irreducible boundary divisors in $\Ms$. Generically, they parametrize the following stable Persson surfaces.
\begin{enumerate}
    \item Type $0$ degenerations with four $A_1$-singularities, i.e., type 3.1 in \cite[Table~1]{alexeev2012non}.
    \item Type \rom{1} degenerations.
\end{enumerate}
\end{remark}

\begin{remark}\label{Rmk:seven_vs_eight}
    In the case of seven lines on $\mathbf{P}^2$, also with the weights $\mathbf{b}=(\frac12,\ldots,\frac12)$, the KSBA compactification $\oM_{\mathbf{b}}(\mathbf{P}^2,7)$ is isomorphic to the GIT compactification $(\mathbf{P}^2)^7\sslash \PGL(\mathbb{C},3)$ with respect to the standard democratic linearization; see the proof of \cite[Theorem~3.2]{alexeev2024explicit}. This is because, for seven lines, the numerical conditions for GIT semi-stability and log canonicity are the same. Moreover, the ambient variety $\mathbf{P}^2$ only degenerates into itself, thus, by Hacking's \cite[Theorem~3.12]{hacking2004compact} and $\mathbf{T}_{\mathbf{P}^2}^2=0$, one knows that $\oM_{\mathbf{b}}(\mathbf{P}^2,7)$ is smooth. Once the line $D_8$ is added, by \cref{Lem:lc_arrangements}, a point through five lines is no longer log canonical, but still GIT semi-stable (actually stable).
\end{remark}
\begin{proposition}\label{Prop:Morphism_KSBA_to_GIT}
    For the eight line arrangements with weight $\mathbf{b}=(\frac12,\ldots,\frac12)$, there is a birational morphism from the KSBA compactification to the GIT compactification with the standard choice of the linearization
\[
    f\colon\oM_{\mathbf{b}}(\mathbf{P}^2,8)\to(\mathbf{P}^2)^8\sslash \PGL(\mathbb{C},3).
\]
\end{proposition}
\begin{proof}
    Observe that for $d=3,n=8$, numerically, the GIT semi-stability condition \cite[Theorem~11.2]{dolgachev2003lectures} is equivalent to the log canonical condition in \cref{Lem:lc_arrangements} with weights $\mathbf{b}^{\prime}=(2/5,\ldots,2/5)$. Then there exists a family of $\mathbf{b}'$-weighted log canonical line arrangements over $(\mathbf{P}^2)^8\sslash \PGL(\mathbb{C},3)$. Each fiber has at least four lines in general linear position. So the $\PGL(\mathbb{C},3)$-action is free, thus the GIT compactification is smooth, which has to be isomorphic to $\oM_{\mathbf{b}'}(\mathbf{P}^2,8)$.
    Moreover, there is no other wall on the linear segment connecting $\mathbf{b},\mathbf{b}^{\prime}$. So one has a morphism 
\[
    \oM_{\mathbf{b}}(\mathbf{P}^2,8)\cong\oM_{\mathbf{b}^{\prime}+\epsilon\mathbf{1}}(\mathbf{P}^2,8)\to\oM_{\mathbf{b}^{\prime}}(\mathbf{P}^2,8)\cong(\mathbf{P}^2)^8\sslash \PGL(\mathbb{C},3)
\]
according to the properties of wall crossings in \cref{Thm:Alexeev_wall_crossing}. The second isomorphism is because the two spaces are both smooth. For the GIT quotient, the smoothness comes from the freeness of the $\PGL(\mathbb{C},3)$-action. For the moduli of stable $\mathbf{b}^{\prime}$-weighted line arrangements, since $\mathbf{P}^2$ never breaks, one can apply \cite[Theorem~3.12]{hacking2004compact} again. All three $H^i(\mathcal{T}_Y^j)=0,i+j=2$ are clear for $Y\cong\mathbf{P}^2$, which implies that any stable $(Y,\mathbf{b}^{\prime}D)$ is unobstructed by the local-global spectral sequence.
\end{proof}

\begin{remark}
     A smooth Persson surface is a double cover of a smooth Campedelli surface; see \cref{Rmk:surfaces_in_the_middle}\cref{Item:1_rmk_surfaces_in_middle}. The compactified moduli and stable degenerations of Campedelli surfaces are already given in \cite[Section~3]{alexeev2024explicit}, similar to \cref{Thm:Persson_KSBA} but with $\mathbf{P}^2$ never degenerating to a reducible variety. In terms of the language of polytope tilings, as shown in \cite[Section~5.2]{alexeev2024explicit}, the $\mathbf{b}$-cut $\Delta_{\mathbf{b}}(3,7)$ (see \cref{Def:b-cut}) admits only the trivial matroid tiling.
\end{remark}

\subsection{A morphism to the minimal compactification}\label{Sec:oudompheng}

We first review the relevant parts of Oudompheng’s thesis \cite{oudomphengthesis} on $D_{1,6}$-polarized Enriques surfaces.

Let $\mathbb{Z}^{1,6}$ be the unique odd unimodular lattice of signature $(1,6)$, and $D_{1,6}$ be the index $2$ sublattice of $\mathbb{Z}^{1,6}$ with elements self product even, or equivalently, the sum of the coordinates is even. In $D_{1,6}$, fix seven distinguished vectors: $2e_0,e_1\pm e_2,e_3\pm e_4,e_5\pm e_6$, where $e_0,\ldots,e_6$ is the standard basis of $\mathbb{Z}^{1,6}$.

\begin{definition} [{\cite[Definition 5.\rom1]{oudomphengthesis}}]\label{Def:d16_pol_Enriques}
    A \emph{$D_{1,6}$-polarized Enriques surface} $S$ is an Enriques surface with a primitive lattice embedding $D_{1,6}\hookrightarrow\Pic(S)$ such that:
    \begin{enumerate}
        \item $2e_0$ corresponds to a big and nef divisor class $H$ with $H^2=4$.
        \item The other distinguished vectors $e_1\pm e_2,e_3\pm e_4,e_5\pm e_6$ correspond to six irreducible $(-2)$-curves (isomorphic to $\mathbf{P}^1$) $R_1^{\pm},R_2^{\pm},R_3^{\pm}$ respectively. Under the bidouble cover, they are the preimages of three toric boundary
        $\mathbf{P}^1$'s of $\Bl_3\mathbf{P}^2$; see the three black thick lines on the boundary of the hexagon in \cref{Fig:hexagon}.
    \end{enumerate}
\end{definition}

As discussed in \cref{Rmk:surfaces_in_the_middle}, there are $28$ intermediate $D_{1,6}$-polarized Enriques surfaces, each with six $A_1$-singularities, which are, under the cover map, the preimages of the points $l_i\cap l_i'$, $i=1,2,3$ (the black dots of the right picture in  \cref{Fig:hexagon}). Blow up the three points on $\mathbf{P}^2$ then take the bidouble cover, one obtains smooth $D_{1,6}$-polarized Enriques surfaces as in \cref{Def:d16_pol_Enriques}.

\begin{center}
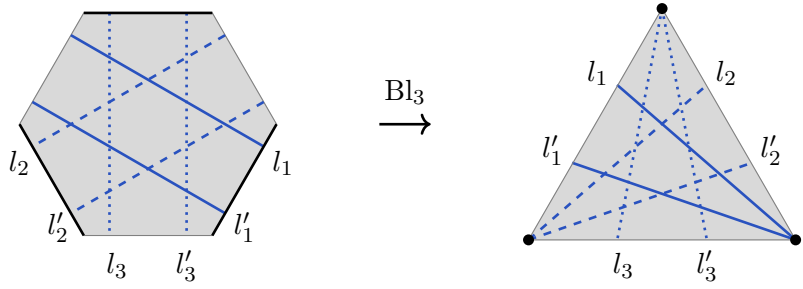
\begin{figure}[htbp]
\begin{tikzpicture}[line width=1pt, scale=1.7]
\begin{scope}[xshift=4cm,yshift=-0.3cm]
    \newdimen\R
    \R=1.2cm
    \draw [gray, thin, fill=gray, opacity=0.3] (90:\R)    \foreach \x in {90,210,330,90} {  -- (\x:\R) };
    \draw [SkyBlue, dotted] (90:\R) -- (240:0.7);
    \draw [SkyBlue, dotted] (90:\R) -- (300:0.7);
    \draw [SkyBlue, dashed] (210:\R) -- (0:0.7);
    \draw [SkyBlue, dashed] (210:\R) -- (60:0.7);
    \draw [SkyBlue] (-30:\R) -- (120:0.7);
    \draw [SkyBlue] (-30:\R) -- (180:0.7);
    \draw [gray, thin] (90:\R)    \foreach \x in {90,210,330,90} {  -- (\x:\R) };
    \node at (90:\R) [circle, fill=black, scale=0.4]{};
    \node at (210:\R) [circle, fill=black, scale=0.4]{};
    \node at (330:\R) [circle, fill=black, scale=0.4]{};
    \node at (-0.3,-0.8) {$l_3$};
    \node at (0.35,-0.8) {$l_3'$};
    \node at (0.85,0.1) {$l_2'$};
    \node at (0.5,0.7) {$l_2$};
    \node at (-0.85,0.1) {$l_1'$};
    \node at (-0.5,0.7) {$l_1$};
\end{scope}
\begin{scope}[xshift=2cm]
    \draw[arrows=->,line width=1pt, black] (-0.2,0)--(0.2,0);
    \node at (0,0.25) {\color{black}$\Bl_3$};
\end{scope}
\begin{scope}
    \newdimen\R
    \R=1cm
    \draw [gray, thin] (0:\R)    \foreach \x in {60,120,...,360} {  -- (\x:\R) };
    \draw [gray, thin, fill=gray, opacity=0.3] (0:\R)    \foreach \x in {60,120,...,360} {  -- (\x:\R) };
    \draw [SkyBlue, rotate=60] (-0.3,0.866) -- (-0.3,-0.866);
    \draw [SkyBlue, rotate=60] (0.3,0.866) -- (0.3,-0.866);
    \draw [SkyBlue, dotted] (-0.3,0.866) -- (-0.3,-0.866);
    \draw [SkyBlue, dotted] (0.3,0.866) -- (0.3,-0.866);
    \draw [SkyBlue, dashed, rotate=-60] (-0.3,0.866) -- (-0.3,-0.866);
    \draw [SkyBlue, dashed, rotate=-60] (0.3,0.866) -- (0.3,-0.866);
    \node at (-0.7,-0.8) {$l_2'$}; 
    \node at (-1,-0.3) {$l_2$}; 
    \node at (-0.24,-1.1) {$l_3$};
    \node at (0.3,-1.1) {$l_3'$}; 
    \node at (0.75,-0.8) {$l_1'$};
    \node at (1.05,-0.3) {$l_1$}; 
    \draw [black] (1,0) -- (-60:1);
    \draw [black] (-1,0) -- (-120:1);
    \draw [black] (60:1) -- (120:1);
\end{scope}
\end{tikzpicture}
\caption{Branch divisors on $\Bl_3\mathbf{P}^2$ and $\mathbf{P}^2$ for smooth and singular $D_{1,6}$-polarized Enriques surfaces.}
\label{Fig:hexagon}
\end{figure} 
\end{center}

\begin{lemma}[{\cite[Proposition 6.42]{oudomphengthesis}}]\label{Lem:oudompheng_6.42}
    The boundary structure of the Baily--Borel compactification $\MBB_{D_{1,6}}$ of the moduli of $D_{1,6}$-polarized Enriques surfaces is given below as in \cref{Fig:BB_boundary} and fibers as in \cref{Fig:line_degenerations_over_BB_boundaries}.
    \begin{enumerate}[label=\textup{(\arabic*)}]
        \item Three $0$-cusps: $p_{\even}$, and $q_1,q_2$.
        \item Two $1$-cusps: $C_{\even}$ connects $p_{\even}$ and $q_1$, $C_{\odd}$ connects $q_1,q_2$.
    \end{enumerate}
\begin{center}
\begin{figure}[htpb]
\begin{tikzpicture}[scale=0.6, line width=1pt]
    \draw (-5,0) to [out=45,in=135] (1,-1);
    \draw (5,0) to [out=135,in=45] (-1,-1);
    \node [circle, fill=black, scale=0.4, label=below:{$q_1$}] at (0,-0.2) {};
    \node [circle, fill=black, scale=0.4, label={[xshift=0.3cm]below:{$p_{\even}$}}] at (-4.5,0.4) {};
    \node [circle, fill=black, scale=0.4, label=below:{$q_2$}] at (4.5,0.4) {};
    \node [label=below:{$C_{\even}$}] at (-3,2.2) {};
    \node [label=below:{$C_{\odd}$}] at (3,2.2) {};
\end{tikzpicture}
\caption{Boundary of the Baily--Borel compactification of the moduli of $D_{1,6}$-polarized Enriques surfaces.}\label{Fig:BB_boundary}
\end{figure}
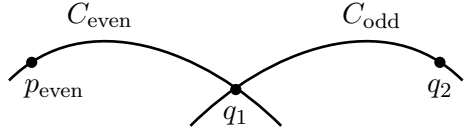
\end{center}
\end{lemma}

\begin{center}
\begin{figure}[htbp]
\begin{tikzpicture}[scale=0.4, line width=1pt]
\begin{scope}[xshift=14cm,yshift=-6cm]
    \draw [SkyBlue] (-3,-2.05) -- (3,-2.05);
    \draw [SkyBlue] (-3,-1.95) -- (3,-1.95);
    \draw [SkyBlue, dashed] (-3.1,-3) -- (0.9,1);
    \draw [SkyBlue, dashed] (-2.95,-3) -- (1.05,1);
    \draw [SkyBlue, dotted] (3.1,-3) -- (-0.9,1);
    \draw [SkyBlue, dotted] (1.95,-3) -- (0.5,1);
    \node at (0,-4) {$\mathcal{X}_{q_2}$};
\end{scope}
\begin{scope}[xshift=7cm]
    \draw [SkyBlue] (-3,-1.95) -- (3,-1.95);
    \draw [SkyBlue] (-3,-2.05) -- (3,-2.05);
    \draw [SkyBlue, dashed] (-2.95,-3) -- (1.05,1);
    \draw [SkyBlue, dotted] (3.1,-3) -- (-0.9,1);
    \draw [SkyBlue, dotted] (1.95,-3) -- (0.5,1);
    \draw [SkyBlue, dashed] (-1.95,-3) -- (-0.5,1);
    \node at (0,-4) {$\mathcal{X}_p,p\in C_{\odd}\backslash\{q_1,q_2\}$};
\end{scope}
\begin{scope}[yshift=-6cm]
    \draw [SkyBlue] (-3,-2.05) -- (3,-2.05);
    \draw [SkyBlue] (-3,-1.95) -- (3,-1.95);
    \draw [SkyBlue, dotted] (-3.1,-3) -- (0.9,1);
    \draw [SkyBlue, dashed] (-2.95,-3) -- (1.05,1);
    \draw [SkyBlue, dashed] (-1.95,-3) -- (-0.5,1);
    \draw [SkyBlue, dotted] (1.95,-3) -- (0.5,1);
    \node at (0,-4) {$\mathcal{X}_{q_1}$};
\end{scope}
\begin{scope}[xshift=-7cm]
    \draw [SkyBlue] (-3,-1.95) -- (3,-1.95);
    \draw [SkyBlue, dotted] (-3.1,-3) -- (0.9,1);
    \draw [SkyBlue, dashed] (-2.95,-3) -- (1.05,1);
    \draw [SkyBlue] (3.1,-3) -- (-0.9,1);
    \draw [SkyBlue, dotted] (1.95,-3) -- (0.5,1);
    \draw [SkyBlue, dashed] (-1.95,-3) -- (-0.5,1);
    \node at (0,-4) {$\mathcal{X}_p,p\in C_{\even}\backslash\{p_{\even},q_1\}$};
\end{scope}
\begin{scope}[xshift=-14cm, yshift=-6cm]
    \draw [SkyBlue] (-3,-2.05) -- (3,-2.05);
    \draw [SkyBlue, dashed] (-3,-1.95) -- (3,-1.95);
    \draw [SkyBlue, dotted] (-3.1,-3) -- (0.9,1);
    \draw [SkyBlue, dashed] (-2.95,-3) -- (1.05,1);
    \draw [SkyBlue] (3.1,-3) -- (-0.9,1);
    \draw [SkyBlue, dotted] (1.95,-3) -- (0.5,1);
    \node at (0,-4) {$\mathcal{X}_{p_{\even}}$};
\end{scope}
\end{tikzpicture}
\caption{Degenerated 6-line arrangements over the the boundary of the Baily--Borel (also the GIT) compactification of the moduli of $D_{1,6}$-polarized Enriques surfaces. The solid lines, dashed lines and dotted lines are different rulings as in \cref{Fig:hexagon}.}\label{Fig:line_degenerations_over_BB_boundaries}
\end{figure}
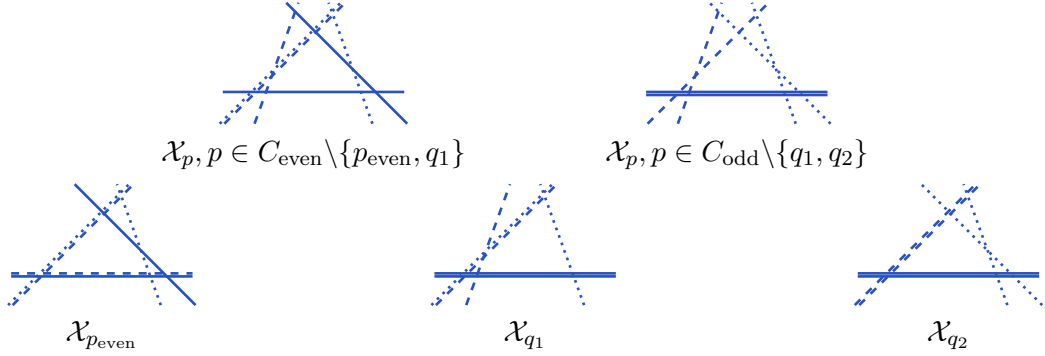
\end{center}

\begin{lemma}[{\cite[Corollary 6.21]{oudomphengthesis}}]\label{lem:git_interpretation_of_BB}
    $\MBB_{D_{1,6}}$ admits a GIT interpretation:
\[
    \MBB_{D_{1,6}}\cong (\Gr(3,6)\sslash T)/(W_3\times\mathbb{Z}/2\mathbb{Z}),
\]
where $W_3=(\mathbb{Z}/2\mathbb{Z})^3\rtimes\mathfrak{S}_3$ and the GIT quotient is in terms of the standard democratic $T$-linearization.
\end{lemma}

Eight lines in general linear position determine $28$ $D_{1,6}$-polarized Enriques surfaces, thus there is a morphism $\M(\mathbf{P}^2,8)\to\prod_{i=1}^{28}\M_{i,D_{1,6}}$. The $G_{\Stab}$-action induces a morphism from the moduli of smooth Persson surfaces $\M=\M(\mathbf{P}^2,8)/G_{\Stab}$  to $\prod_{i=1}^{28}\M_{i,D_{1,6}}/G_{\Stab}$.

\begin{definition}\label{Def:BB_for_Persson}
    The \emph{minimal compactification} of the moduli of smooth Persson surfaces is defined as the closure of the image of $\M$ in $(\prod_{i=1}^{28}\MBB_{i,D_{1,6}})/G_{\Stab}$, denoted by $\MBB$.
\end{definition}
Here we call the compactification ``minimal'' as it is defined by using the Baily--Borel construction.
\begin{proposition}\label{Prop:KSBA_to_BB}
    There exists a surjective morphism $\psi\colon\Ms\to\MBB$. 
\end{proposition}
\begin{proof}
    We interpret $\MBB$ as a product the GIT quotients, each followed by a finite group quotient according to \cref{lem:git_interpretation_of_BB}. Denote by $\mathbf{b}(k)$ the weight $(\frac12,\ldots,\frac12)$ with $k$ coordinates and $T_k$ the torus $\mathbb{G}_m^k/\diag\mathbb{G}_m$. According to \cite[Lemma~3.18]{alexeev2023stable} or \cite[Theorem~7.3]{giansiracusa2014kapranov}, there is a natural morphism $\oM_{\mathbf{b}(8)}(\mathbf{P}^2,8)\to\oM_{\mathbf{b}(7)}(\mathbf{P}^2,7)$ by forgetting one marked weighted (broken) line, since the image does not depend on how to approach a point $p\in\oM_{\mathbf{b}(8)}(\mathbf{P}^2,8)$; for example, see the proof of \cite[Theorem~5.4]{gallardo2024explicit} for a similar argument. By \cite[Theorem~1]{alexeev2024explicit} and \cref{Rmk:seven_vs_eight},  $\oM_{\mathbf{b}(7)}(\mathbf{P}^2,7)$ is isomorphic to the standard GIT quotient $\Gr(3,7)\sslash T_7$. There is also a morphism $\Gr(3,7)\sslash T_7\to\Gr(3,6)\sslash T_6$ since the GIT semi-stability condition is preserved when one line is removed by \cite[Theorem~11.2]{dolgachev2003lectures}. Then the morphism in the theorem is obtained by descending to the finite group $G_{\Stab}$-quotients.
\end{proof}

\begin{remark}
The fibers of $\psi$ can be seen when lifted to $\oM_{\mathbf{b}}(\mathbf{P}^2,8)$ over each factor. For example, take six lines from the eight ones. Then the six lines are divided into three pairs naturally. There are two extra lines added on degenerations in \cref{Lem:oudompheng_6.42} and  \cref{Fig:line_degenerations_over_BB_boundaries}. Fibers over each degeneration are given by their KSBA stable replacements.
\end{remark}

\section{Global Torelli problem with respect to Pardini's period map}\label{Sec:global_Torelli}

This section is devoted to investigating the generic global Torelli problem for Persson surfaces. While our approach parallels that of \cite{pearlstein2019generic}, a key distinction is that no K3 surfaces appear within the covering tower of a Persson surface (see \cref{Rmk:surfaces_in_the_middle}). To address this, we instead utilize \'etale double covers of Persson surfaces, as detailed in \cref{Sec:auxiliary_surfaces}. In \cref{Sec:proof_torelli}, we demonstrate that the Hodge structure on the anti-invariant part is determined by four degree $2$ K3 surfaces. Building on this, we introduce a $\widetilde{G}$-period map and establish that it has generic degree $2$ for the abelian group $\widetilde{G}=(\mathbb{Z}/2\mathbb{Z})^5$.

\subsection{The auxiliary surfaces via \'etale double covers}\label{Sec:auxiliary_surfaces}

To establish our global Torelli theorem, we construct in this section a family of surfaces which are \'etale double covers of Persson surfaces and are equipped with geometric data encoding the appropriate Hodge structures.

In \cite{pearlstein2019generic}, Pearlstein and Zhang established the global Torelli theorem for a family of special Horikawa surfaces with respect to the $(\mathbb{Z}/2\mathbb{Z})^2$-period map (see \cref{Example:special_horikawa} for the construction). In their setting, the global Torelli property is essentially derived from two intermediate K3 surfaces, realized as double covers of $\mathbf{P}^2$ branched along the quintic curve and a line. In our setting of Persson surfaces, while the intermediate surfaces of the $G$-covers (see \cref{Rmk:surfaces_in_the_middle}) do not yield these K3 surfaces, new K3 surfaces arise as intermediate covers between certain \'etale double covers of Persson surfaces and $\mathbf{P}^2$. This motivates a formulation of our global Torelli theorem that mirrors the case of Enriques surfaces, which utilizes their associated K3 \'etale double covers. 

Let $\pi:X\to \mathbf{P}^2$ be a Persson surface. Choose a $2$-torsion line bundle $L\in\Pic(X)[2]$. As discussed in \cref{Lem:Picard} and \cref{Rmk:torsion_L_vs_four_pairs}\cref{Item:1_Rmk:torsion_L_vs_four_pairs}, there are seven choices for such an $L$; each $L$ corresponds to a partition of the eight branch lines into four pairs $D_{g_i},D_{g_i'}$ with $i=1,\ldots,4$, and $L$ is $\mathcal{O}_{X}(\pi^{-1}D_{g_i}-\pi^{-1}D_{g_i'}),i=1,\ldots,4$. A choice of $L\in\Pic(X)[2]$ induces an \'etale double cover of a Persson surface $X$, denoted by $\pi_L:Z_L\to X$. The surface $Z_L$ is a cover of $\mathbf{P}^2$, as a $G$-cover followed by a double cover. But we can obtain more:

\begin{proposition}\label{Prop:composition_is_abelian}
    The cover $\widetilde{\pi}:Z_L\to \mathbf{P}^2$, obtained by composing the $G$-cover $\pi:X\to\mathbf{P}^2$ and the \'etale double cover $\pi_L:Z_L\to X$, is also an abelian cover, with respect to the group $\widetilde{G}=(\mathbb{Z}/2\mathbb{Z})^5$. In particular, the line bundles of the building data can be given as follows.
\begin{equation}\label{Eqn:linebundles_32_to_1}
    L_{\widetilde{\chi}}=\begin{cases}
        \mathcal{O}_{\mathbf{P}^2}(4), \text { if }\widetilde{\chi}=\widetilde{\chi}_0=(\chi_0,0)=(1,0,0,0,0)\\
        \mathcal{O}_{\mathbf{P}^2}(3), \text { if }\widetilde{\chi}=\widetilde{\chi}_i,i=1,\ldots,4\\
        \mathcal{O}_{\mathbf{P}^2}(2), \text { if }\widetilde\chi\text{ is one of the other }22\text{ characters }\\
        \mathcal{O}_{\mathbf{P}^2}(1), \text { if }\widetilde{\chi}=\widetilde{\chi}_i,i=5,\ldots,8\\
        \mathcal{O}_{\mathbf{P}^2}, \text{ if }\widetilde{\chi}=(0,0,0,0,0)
    \end{cases}
\end{equation}
    where
\begin{equation}\label{Eqn:tilde_characters}
\begin{aligned}
    \widetilde{\chi}_1&=(0,1,1,0,1),  \widetilde{\chi}_2=(1,1,0,0,1),  \widetilde{\chi}_3=(1,1,0,1,1),  \widetilde{\chi}_4=(1,1,1,0,1),\\
    \widetilde{\chi}_5&=(0,1,0,0,1),  \widetilde{\chi}_6=(0,1,0,1,1),  \widetilde{\chi}_7=(0,1,1,1,1),  \widetilde{\chi}_8=(1,1,1,0,1).
\end{aligned}
\end{equation}
Moreover, the four line bundles $L_{\widetilde{\chi}_i}$, $i=1,\ldots,4$ determine a partition of the eight lines into four pairs.
\end{proposition}
\begin{proof}
    We prove this by explicitly giving the abelian $\widetilde{G}$-cover building data. Consider the short exact sequence of abelian groups
\begin{align*}
    0\rightarrow G&\rightarrow\widetilde{G}\rightarrow\mathbb{Z}/2\mathbb{Z}\rightarrow 0,\\
    g&\mapsto (g,0).
\end{align*}
Then building data for a $\widetilde{G}$-cover can be constructed so that its restriction to the first four coordinates coincides with the building data for $X$ in \cref{Def:Persson}. For example, take
\begin{align*}
    \widetilde{g}_1&=(1,0,0,0,0),\quad\widetilde{g}_2=(1,0,0,1,1),\quad\widetilde{g}_3=(1,0,1,0,0),\quad\widetilde{g}_4=(1,0,1,1,0),\\
    \widetilde{g}_5&=(1,1,0,0,1),\quad\widetilde{g}_6=(1,1,0,1,0),\quad\widetilde{g}_7=(1,1,1,0,1),\quad\widetilde{g}_8=(1,1,1,1,1).
\end{align*}
    Write $\widetilde{g}_i$ as $(g_i,x_5)$, $x_5\in\mathbb{F}_2$ for $i=1,\ldots,8$, where the eight $g_i$'s label branch lines of the $G$-cover $X\to\mathbf{P}^2$. To construct the $\widetilde{G}$-cover, we choose the same branch divisors as for the $G$-cover, i.e., $D_{\widetilde{g}_i}=D_{g_i}$ for $i=1,\ldots,8$, and all other $D_{\widetilde{g}}=\emptyset$. From the fundamental relation \cref{Item:3_fundamental_relation}, one can solve all the $32$ line bundles of the building data as in \cref{Eqn:linebundles_32_to_1}.

    According to \cite[Theorem, Page~191]{pardini1991abelian}, the existence of such building data guarantees that the resulting $\widetilde{G}$-cover, denoted by $\widetilde{\pi}\colon Z\to\mathbf{P}^2$, is indeed an abelian cover. Because both $Z$ and $X$ are abelian covers of $\mathbf{P}^2$ branched along the exact same eight lines, the induced morphism $Z\to X$ is \'etale, meaning that $Z$ must be isomorphic to $Z_L$ for some $2$-torsion line bundle $L$. Moreover, $L\in \Pic(X)[2]$ is uniquely determined by the characters $\widetilde{\chi}_i$, $i=1,\ldots,4$, via the partition described in \cref{Rmk:torsion_L_vs_four_pairs}\ref{Item:1_Rmk:torsion_L_vs_four_pairs}. For instance, after suitably rearranging the labels of the branch data, we obtain:
\begin{equation}\label{Eqn:relabel_g_i_g_i'}
\begin{aligned}
    g_1&=(1,0,0,0),\quad g_2=(1,1,0,1),\quad g_3=(1,0,1,1),\quad g_4=(1,0,1,0),\\
    g_1'&=(1,1,0,0),\quad g_2'=(1,0,0,1),\quad g_3'=(1,1,1,1), \quad g_4'=(1,1,1,0).
\end{aligned}
\end{equation}
    The new indices $\widetilde{g}_i,\widetilde{g}_i'$ are uniquely determined by the original $\widetilde{g}_i$, $i=1,\ldots,8$, such that $\widetilde{\chi}_i(\widetilde{g}_j)=\widetilde{\chi}_i(\widetilde{g}_j')=(-1)^{\delta_{ij}}$, where $\delta_{ij}$ is the Kronecker delta. Explicitly, the relabeled elements are given by:
\begin{equation}\label{Eqn:relabel_g_i_g_i'_with_tilde}
\begin{aligned}
    \widetilde g_1&=(1,0,0,0,0), & \widetilde g_2&=(1,1,0,1,0), & \widetilde g_3&=(1,0,1,1,0), & \widetilde g_4&=(1,0,1,0,0),\\
    \widetilde g_1'&=(1,1,0,0,1), & \widetilde g_2'&=(1,0,0,1,1), & \widetilde g_3'&=(1,1,1,1,1), & \widetilde g_4'&=(1,1,1,0,1).
\end{aligned}
\end{equation}
This completes the proof.
\end{proof}

\begin{remark}\label{Lem:counting_label_5_coord}
    There are totally $112$ different ways to label the last coordinates of $\widetilde{g}\in\widetilde{G}$ for the eight lines $D_{\widetilde{g}}$, such that, in the resulting $\widetilde{G}$-building data, there exist four $L_{\widetilde{\chi}}$ isomorphic to $\mathcal{O}_{\mathbf{P}^2}(3)$.
    For one choice of $L\in\Pic(X)[2]$, we know that the \'etale double cover induced by it is also abelian with respect to $\widetilde{G}$ by \cref{Prop:composition_is_abelian}. The argument works for all seven $L\in\Pic(X)[2]$ by symmetry. Each $L$ corresponds to a partition of the eight element set $\{g\in G\mid D_g\neq\emptyset\}$ into four pairs $g_i,g_i',i=1,\ldots,4$ by \cref{Rmk:torsion_L_vs_four_pairs}\cref{Item:1_Rmk:torsion_L_vs_four_pairs}. For each partition,  the labels within each pair can be interchanged, hence there are $7\times 2^4=112$ choices in total.
\end{remark}

\begin{remark}\label{Rmk:from_Pardini}
    The fact that $Z_L\to X$ is abelian can also be seen as follows. By \cite[Example~1]{alexeevpardini2013exist}, there is an abelian $G_{\max}=(\mathbb{Z}/2\mathbb{Z})^7$-cover $X_{\max}\cong \mathbf{P}^2$ of $Y=\mathbf{P}^2$ branched along the eight lines, which factors through $X$. The cover $X_{\max}\to X$ is \'etale because they have the same branch divisors as covers over $\mathbf{P}^2$. Notice that $X_{\max}\to X$ is the universal cover. Thus, the \'etale cover $Z_L\to X$ is an intermediate cover $X_{\max}\to Z_L\to X$. So $Z_L\to Y=\mathbf{P}^2$ is also abelian since $X_{\max}\to Y$ is Galois.
\end{remark}

The numerical invariants of $Z_L$ are computed in the following lemma.
\begin{lemma} \label{Lem:numerical_data_of_Z_L}
    We have $K_{Z_L}^2=32,p_g(Z_L)=7,q(Z_L)=0$ and $h^{1,1}(Z_L)=48$. Moreover, the lattice on the torsion free part $H^2_{\f}(Z_L,\mathbb{Z})$ of $H^2(Z_L,\mathbb{Z})$ with respect to the intersection form is isometric to $U^{\oplus 15}\oplus E_8(-1)^{\oplus 4}$. 
\end{lemma}
\begin{proof}
    The numerical invariants of $Z_L$ can be easily calculated using those of the Persson surface $X$ computed in \cref{Prop:numerical_data_of_Persson}. Specifically, we have $K_{Z_L}=\pi_L^*K_X$ which implies $K_{Z_L}^2=2K_X^2=32$. The irregularity $q(Z_L)=0$ as $h^1(\mathcal{O}_{Z_L})=\sum_{\widetilde{\chi}\in\widetilde{G}^*}h^1(L^{-1}_{\widetilde{\chi}})$, which is
\[
    h^1(\mathcal{O}_{\mathbf{P}^2})+4h^1(\mathcal{O}_{\mathbf{P}^2}(-1))+22h^1(\mathcal{O}_{\mathbf{P}^2}(-2))+4h^1(\mathcal{O}_{\mathbf{P}^2}(-3))+h^1(\mathcal{O}_{\mathbf{P}^2}(-4))=0
\]
    according to \cref{Eqn:linebundles_32_to_1}.
    We also note that $\chi(\mathcal{O}_{Z_L})=2\chi(\mathcal{O}_X)=8$. As a result, the geometric genus $p_g(Z_L)=8-1=7$. (Or apply \cite[Proposition~4.1.c]{pardini1991abelian}: $h^{2,0}(Z_L)=h^0(\mathcal{O}_{\mathbf{P}^2}(K_{\mathbf{P}^2})\otimes L_{\widetilde{\chi}_0})+\sum_{i=1}^4 h^0(\mathcal{O}_{\mathbf{P}^2}(K_{\mathbf{P}^2})\otimes L_{\widetilde{\chi}_i})=3+4=7$.) Since $\chi_{\mathrm{top}}(Z_L)=2\chi_{\mathrm{top}}(X)=64$, we have $h^{1,1}(Z_L)=64-2-2p_g(Z_L)=48$. By the same arguments in the proof of \cref{Prop:numerical_data_of_Persson}, the unimodular lattice on $H^2_{\f}(Z_L,\mathbb{Z})$ is even. The index can be computed by Hirzebruch index theorem $\tau(Z_L)=(K_{Z_L}^2-2\chi_{\mathrm{top}}(Z_L))/3=-32$, and thus the signature of $H^2_{\f}(Z_L,\mathbb{Z})$ is $(15,47)$. Then the lattice is $U^{\oplus 15}\oplus E_8(-1)^{\oplus 4}$ according to the classification of indefinite even unimodular lattices.
\end{proof}

To formulate our global Torelli theorem with respect to the $\widetilde{G}$-period map, we consider the Hodge structure on
\[
    H^2(X,\mathbb{Z})^\perp_L:=\text{the anti-invariant part of }H^2_\mathrm{f}(Z_L,\mathbb{Z})\text{ under the involution},
\]
whose Hodge numbers are $[4,24,4]$ by \cref{Lem:numerical_data_of_Z_L}. We simply write  $H^2(X,\mathbb{Z})^\perp$ for $ H^2(X,\mathbb{Z})^\perp_L$ when no confusion can arise.

\begin{lemma}\label{Lem:lattice_on_H_perp}
    Equipped with the intersection form, $H^2(X,\mathbb{Z})^\perp$ is isometric to the lattice $\Lambda^-=U\oplus U(2)^{\oplus 7}\oplus E_8(-2)^{\oplus 2}$.
\end{lemma}
\begin{proof}
Let $\Lambda=U^{\oplus 15}\oplus E_8(-1)^{\oplus4}$. Consider the involution $\iota$ on $\Lambda$ given by 
\[
    (x_1,x_2,\ldots,x_{13},x_{14},x_{15},y_1,y_2,y_3,y_4)\mapsto (x_2,x_1,\ldots,x_{14},x_{13},-x_{15},y_2,y_1,y_4,y_3).
\]
Then the $\iota$-invariant sublattice is $\Lambda^+\cong U(2)^{\oplus 7}\oplus E_8(-2)^{\oplus 2}$ and the $\iota$-anti-invariant sublattice is $\Lambda^-\cong U\oplus U(2)^{\oplus 7}\oplus E_8(-2)^{\oplus 2}$. On the other hand, we also have that $\pi_L^* H^2_{\f}(X,\mathbb{Z})$ is a full rank sublattice of $H^2_{\f}(Z_L,\mathbb{Z})^{\sigma}$, which is the invariant part of $H^2_{\f}(Z_L,\mathbb{Z})^{\sigma}$ induced by the involution $\sigma$ associated with $\pi_L\colon Z_L\to X$. By \cite[Proposition~A.6.2]{peterssterk}, the index of $\pi^* H^2_{\f}(X,\mathbb{Z})$ in $H^2_{\f}(Z_L,\mathbb{Z})$ is $1$, therefore we have $\Lambda^+\cong H^2_{\f}(Z_L,\mathbb{Z})^{\sigma}$. According to \cite[Proposition~15.2.10]{peterssterk}, there exists an isometry between $H^2_{\f}(Z_L,\mathbb{Z})$ and $\Lambda$ that intertwines the involutions $\sigma$ and $\iota$. Consequently, the anti-invariant parts are also isomorphic to each other $H^2(X,\mathbb{Z})^\perp\cong \Lambda^-\cong U\oplus U(2)^{\oplus 7}\oplus E_8(-2)^{\oplus 2}$. \qedhere
\end{proof}

According to \cref{prop:eigen}, the $32$-dimensional vector space $H^2(X,\mathbb{C})^\perp := H^2(X,\mathbb{Z})^\perp\otimes \mathbb{C}$ admits the following $\widetilde{G}^*$-eigenspace decomposition and geometric interpretation.

\begin{remark}\label{Rmk:ZL_eigen_interpretation}
    From the line bundles in the $\widetilde{G}$-cover building data \cref{Eqn:linebundles_32_to_1} with characters $\widetilde{\chi}_i$, $i=1,\ldots,8$ as in \cref{Eqn:tilde_characters}, we obtain the following $16$ intermediate surfaces between $Z_L$ and $\mathbf{P}^2$ but not listed in \cref{Rmk:surfaces_in_the_middle}\cref{Item:3_rmk_surfaces_in_middle}, \cref{Item:4_rmk_surfaces_in_middle} between $X$ and $\mathbf{P}^2$.
\begin{enumerate}[wide, label=(\arabic*)]
    \item\label{Item:1_Rmk:ZL_eigen_interpretation} There are four K3 surfaces  $S_{\widetilde{\chi}_i}$ associated to $L_{\widetilde{\chi}_i}$, $i=1,\ldots,4$, each with $\binom{6}{2}=15$ $A_1$-singularities. The K3 surface $S_{\widetilde{\chi}_i}$ is a double cover of $\mathbf{P}^2$ branched along the six lines $D_{g_j},D_{g_j'}$, $j\neq i$ in \cref{Eqn:relabel_g_i_g_i'}. These K3 surfaces do not appear as intermediate surfaces between the Persson surface and $\mathbf{P}^2$, as the line bundle $\mathcal{O}_{\mathbf{P}^2}(3)$ does not occur in the building data of the $G$-cover $X\to\mathbf{P}^2$. The primitive middle Hodge numbers for this (singular) K3 surface are $[1,4,1]$. Together, their primitive cohomology span the full $4$-dimensional space of the $(2,0)$-part of $H^2(X,\mathbb{C})^\perp$ and provide a $4\times4=16$-dimensional subspace to the $(1,1)$-part of $H^2(X,\mathbb{C})^\perp$.
\begin{equation}\label{Eqn:partial_cover_tower}
    \begin{tikzcd}
        &Z_L \arrow[dd, "/\widetilde{G}"] \arrow[dl, swap, "/G"] \arrow[dr, "2:1"]\\
        S_{\widetilde{\chi}_i}\arrow[dr, swap, "2:1"] & &X\arrow[dl, "/G"]\\
        &\mathbf{P}^2
    \end{tikzcd}
\end{equation}
    \item\label{Item:2_Rmk:ZL_eigen_interpretation} 
    There are eight del Pezzo surfaces of degree $2$ associated with $L_{\widetilde{\chi}}=\mathcal{O}_{\mathbf{P}^2}(2)$, none of which appear as intermediate surfaces between the Persson surface and $\mathbf{P}^2$. Each is obtained as a double cover of $\mathbf{P}^2$ branched along four lines, and consequently has six $A_1$-singularities. The branch lines are depicted in the second row of \cref{Fig:4g_for_22_dP2}. (In contrast, the first row illustrates $14$ del Pezzo surfaces of degree $2$ which have appeared in the covering tower of the Persson surface, as noted in \cref{Rmk:surfaces_in_the_middle}\cref{Item:3_rmk_surfaces_in_middle}.) Their second primitive cohomology spans an $8$-dimensional subspace within the $(1,1)$-component of $H^2(X,\mathbb{C})^\perp$.
    
\begin{center}
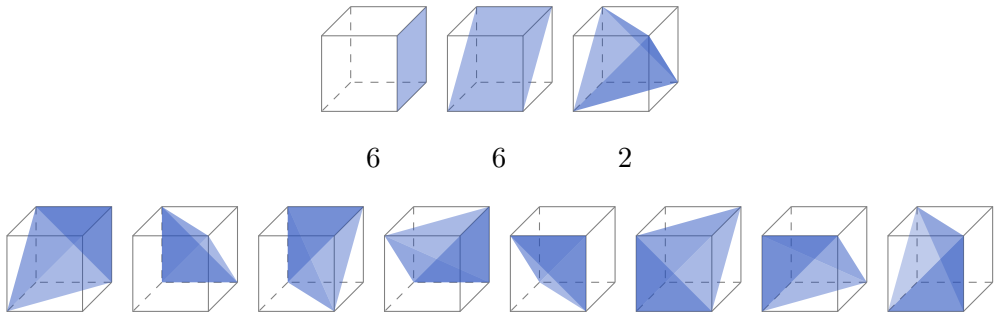
\begin{figure}[htpb]
\begin{minipage}{0.11\textwidth}
\begin{tikzpicture}[color=gray]

\coordinate (A) at (0,0,0);
\coordinate (B) at (1,0,0);
\coordinate (C) at (1,1,0);
\coordinate (D) at (0,1,0);
\coordinate (E) at (0,0,1);
\coordinate (F) at (1,0,1);
\coordinate (G) at (1,1,1);
\coordinate (H) at (0,1,1);

\draw [dashed] (B) -- (A) -- (D);
\draw (D) -- (C) -- (B);

\draw (E)--(F)--(G)--(H)--cycle;

\draw [dashed] (A)--(E);
\draw (B)--(F);
\draw (C)--(G);
\draw (D)--(H);

\fill [SkyBlue, opacity=0.4] (B) -- (C) -- (G) --(F) -- cycle;
\node at (0.3,-1) [color=black] {6};
\end{tikzpicture}
\end{minipage}
\begin{minipage}{0.11\textwidth}
\begin{tikzpicture}[color=gray]

\draw [dashed] (B) -- (A) -- (D);
\draw (D) -- (C) -- (B);

\draw (E)--(F)--(G)--(H)--cycle;

\draw [dashed] (A)--(E);
\draw (B)--(F);
\draw (C)--(G);
\draw (D)--(H);

\fill [SkyBlue, opacity=0.4] (E) -- (F) -- (C) -- (D) -- cycle;
\node at (0.3,-1) [color=black] {6};
\end{tikzpicture}
\end{minipage}
\begin{minipage}{0.11\textwidth}
\begin{tikzpicture}[color=gray]

\draw [dashed] (B) -- (A) -- (D);
\draw (D) -- (C) -- (B);

\draw (E)--(F)--(G)--(H)--cycle;

\draw [dashed] (A)--(E);
\draw (B)--(F);
\draw (C)--(G);
\draw (D)--(H);

\fill [SkyBlue, opacity=0.4] (E) -- (D) -- (G) -- cycle;
\fill [SkyBlue, opacity=0.6] (E) -- (B) -- (G) -- cycle;
\fill [SkyBlue, opacity=0.3] (D) -- (B) -- (G) -- cycle;
\node at (0.3,-1) [color=black] {2};
\end{tikzpicture}
\end{minipage}\\
\vspace{10pt}
\begin{minipage}{0.11\textwidth}
\begin{tikzpicture}[color=gray]

\draw [dashed] (B) -- (A) -- (D);
\draw (D) -- (C) -- (B);

\draw (E)--(F)--(G)--(H)--cycle;

\draw [dashed] (A)--(E);
\draw (B)--(F);
\draw (C)--(G);
\draw (D)--(H);

\fill [SkyBlue, opacity=0.4] (B) -- (C) -- (D) -- cycle;
\fill [SkyBlue, opacity=0.5] (E) -- (C) -- (D) -- cycle;
\fill [SkyBlue, opacity=0.4] (B) -- (C) -- (E) -- cycle;
\end{tikzpicture}
\end{minipage}
\begin{minipage}{0.11\textwidth}
\begin{tikzpicture}[color=gray]

\draw [dashed] (B) -- (A) -- (D);
\draw (D) -- (C) -- (B);

\draw (E)--(F)--(G)--(H)--cycle;

\draw [dashed] (A)--(E);
\draw (B)--(F);
\draw (C)--(G);
\draw (D)--(H);

\fill [SkyBlue, opacity=0.4] (B) -- (A) -- (D) -- cycle;
\fill [SkyBlue, opacity=0.5] (A) -- (G) -- (D) -- cycle;
\fill [SkyBlue, opacity=0.4] (B) -- (A) -- (G) -- cycle;
\end{tikzpicture}
\end{minipage}
\begin{minipage}{0.11\textwidth}
\begin{tikzpicture}[color=gray]

\draw [dashed] (B) -- (A) -- (D);
\draw (D) -- (C) -- (B);

\draw (E)--(F)--(G)--(H)--cycle;

\draw [dashed] (A)--(E);
\draw (B)--(F);
\draw (C)--(G);
\draw (D)--(H);

\fill [SkyBlue, opacity=0.4] (C) -- (A) -- (D) -- cycle;
\fill [SkyBlue, opacity=0.5] (A) -- (F) -- (D) -- cycle;
\fill [SkyBlue, opacity=0.4] (D) -- (F) -- (C) -- cycle;
\end{tikzpicture}
\end{minipage}
\begin{minipage}{0.11\textwidth}
\begin{tikzpicture}[color=gray]

\draw [dashed] (B) -- (A) -- (D);
\draw (D) -- (C) -- (B);

\draw (E)--(F)--(G)--(H)--cycle;

\draw [dashed] (A)--(E);
\draw (B)--(F);
\draw (C)--(G);
\draw (D)--(H);

\fill [SkyBlue, opacity=0.4] (B) -- (A) -- (C) -- cycle;
\fill [SkyBlue, opacity=0.5] (A) -- (B) -- (H) -- cycle;
\fill [SkyBlue, opacity=0.4] (B) -- (C) -- (H) -- cycle;
\end{tikzpicture}
\end{minipage}
\begin{minipage}{0.11\textwidth}
\begin{tikzpicture}[color=gray]

\draw [dashed] (B) -- (A) -- (D);
\draw (D) -- (C) -- (B);

\draw (E)--(F)--(G)--(H)--cycle;

\draw [dashed] (A)--(E);
\draw (B)--(F);
\draw (C)--(G);
\draw (D)--(H);

\fill [SkyBlue, opacity=0.4] (H) -- (G) -- (F) -- cycle;
\fill [SkyBlue, opacity=0.5] (A) -- (G) -- (H) -- cycle;
\fill [SkyBlue, opacity=0.4] (F) -- (A) -- (G) -- cycle;
\end{tikzpicture}
\end{minipage}
\begin{minipage}{0.11\textwidth}
\begin{tikzpicture}[color=gray]

\draw [dashed] (B) -- (A) -- (D);
\draw (D) -- (C) -- (B);

\draw (E)--(F)--(G)--(H)--cycle;

\draw [dashed] (A)--(E);
\draw (B)--(F);
\draw (C)--(G);
\draw (D)--(H);

\fill [SkyBlue, opacity=0.4] (H) -- (E) -- (F) -- cycle;
\fill [SkyBlue, opacity=0.5] (H) -- (E) -- (C) -- cycle;
\fill [SkyBlue, opacity=0.4] (E) -- (F) -- (C) -- cycle;
\end{tikzpicture}
\end{minipage}
\begin{minipage}{0.11\textwidth}
\begin{tikzpicture}[color=gray]

\draw [dashed] (B) -- (A) -- (D);
\draw (D) -- (C) -- (B);

\draw (E)--(F)--(G)--(H)--cycle;

\draw [dashed] (A)--(E);
\draw (B)--(F);
\draw (C)--(G);
\draw (D)--(H);

\fill [SkyBlue, opacity=0.4] (H) -- (E) -- (G) -- cycle;
\fill [SkyBlue, opacity=0.5] (H) -- (G) -- (B) -- cycle;
\fill [SkyBlue, opacity=0.4] (H) -- (E) -- (B) -- cycle;
\end{tikzpicture}
\end{minipage}
\begin{minipage}{0.11\textwidth}
\begin{tikzpicture}[color=gray]

\draw [dashed] (B) -- (A) -- (D);
\draw (D) -- (C) -- (B);

\draw (E)--(F)--(G)--(H)--cycle;

\draw [dashed] (A)--(E);
\draw (B)--(F);
\draw (C)--(G);
\draw (D)--(H);

\fill [SkyBlue, opacity=0.6] (E) -- (F) -- (G) -- cycle;
\fill [SkyBlue, opacity=0.3] (G) -- (D) -- (E) -- cycle;
\fill [SkyBlue, opacity=0.3] (G) -- (F) -- (D) -- cycle;
\end{tikzpicture}
\end{minipage}

\caption{Convex hulls of $x\in\mathbb{F}_2^3$ for $g=(1,x)$ of the four branch lines $D_g$ of the $22$ del Pezzo surfaces of degree $2$. In the first row, the numbers count the blue polytopes on the cube up to symmetry. For example, for the first picture, there are six facets of the cube. The cubes in the second row are for the eight new del Pezzo surfaces, which do not appear in the middle of the $G$-cover $X\to\mathbf{P}^2$. One takes a triangle from either the front square or the back one, then choose the 4-th vertex on the opposite square away from the triangle.}
\label{Fig:4g_for_22_dP2}
\end{figure}
\end{center}

    \item There are four weighted projective planes $\mathbf{P}(1,1,2)$ associated to $L_{\widetilde{\chi}_i}$, $i=5,\ldots,8$ as in \cref{Eqn:tilde_characters}. Each is a double cover of $\mathbf{P}^2$ branched along two lines. They do not contribute to $H^2(X,\mathbb{C})^\perp$. In fact, let $F$ be a ruling of $\mathbf{P}(1,1,2)$, then the pullback of $2F$ represents the ample class of $H^2(Z_L,\mathbb{C})$.
\end{enumerate}
\end{remark}

\begin{remark}
    The stable degeneration data of the $\mathbf{b}$-weighted line arrangement in \cref{Sec:KSBA} can be also used for the moduli of the surfaces $Z_L$. However, the deformation space of $Z_L$ has dimension greater than $8$, by a computation similar to that in \cref{Prop:deformation_dim}.
\end{remark}

\subsection{A generic global Torelli type theorem and its proof}

In this subsection, we introduce our $\widetilde{G}$-period map following \cite{pardini1998period} and \cite{dolgachev2007moduli}; see \cref{Def:our_moduli_functor,Def:G-period_map} and \cref{Notation:eigen_period_domain}. Our main result concerning the generic global Torelli problem is stated in \cref{Thm:degree_2_periodmap}, which establishes that this $\widetilde{G}$-period map is generically of degree $2$. The key ingredient is the global Torelli theorem for the degree $2$ K3 surfaces obtained as double covers of $\mathbf{P}^2$ branched along six lines, based on the results in Matsumoto--Sasaki--Yoshida \cite{matsumoto1992monodromy}. However, in our setting, we need to choose a larger discrete group than the one used in \cite{matsumoto1992monodromy}, which is explained in \cref{Prop:induced_period_map_injective}.

\label{Sec:proof_torelli}
\begin{definition}\label{Def:our_moduli_functor}
    Let $\widetilde{\M}$ be the moduli of pairs $(X,L)$, where $X$ is a smooth Persson surface and $L$ is a $2$-torsion line bundle on it.
\end{definition}

\begin{proposition}\label{Prop:tM_MP2_8_quotient}
    The moduli space $\widetilde\M$ is the quotient of $\M(\mathbf{P}^2,8)$ by the group $(\mathbb{Z}/2\mathbb{Z})^3\rtimes\mathfrak{S}_4$. In particular, $\widetilde\M$ is connected, and it is a cyclic $(\mathbb{Z}/7\mathbb{Z})$-cover of the moduli of smooth Persson surfaces $\M$.
\end{proposition}
\begin{proof}
    The label set of divisors is $\{g=(1,x)\in G\mid D_g\neq\emptyset\}=\{1\}\times \mathbb{F}_2^3$. 
    The group $\mathrm{Aff}(\mathbb{F}_2,3)=(\mathbb{Z}/2\mathbb{Z})^3\rtimes\GL(\mathbb{F}_2,3)<\GL(\mathbb{F}_2,4)$ acts on this label set via
\[
    \begin{pmatrix}
        1 &b\\
        0 &A
    \end{pmatrix}
    \cdot (1,x)=(1,Ax+b),
\]
    where $(b,A)\in(\mathbb{Z}/2\mathbb{Z})^3\rtimes\GL(\mathbb{F}_2,3)$.
    Notice that by \cref{Rmk:torsion_L_vs_four_pairs}\cref{Item:1_Rmk:torsion_L_vs_four_pairs}, any partition of $\mathbb{F}_2^3$ into four pairs associated to a $2$-torsion line bundle can be written as
\[
    \mathbb{F}_2^3=\{0,v\}\sqcup\{v_1,v_1+v\}\sqcup\{v_2,v_2+v\}\sqcup\{v_3,v_3+v\}.
\]
    If the partition is preserved, one obtains
\[
    \{Ax+b,Ax+Av+b\}=\{y,y+v\}
\]
    for any given $x$ and some $y$ in $\{0,v_1,v_2,v_3\}$. Thus 
\[
    (Ax+Av+b)-(Ax+b)=\pm[(y+v)-y]=\pm v,
\]    
    i.e., $Av=v$ since the vector space is over $\mathbb{F}_2$, $v=-v$. The stabilizer of any $v\in\mathbb{F}_2^3\backslash\{0\}$ in $\GL(\mathbb{F}_2,3)$ is $(\mathbb{Z}/2\mathbb{Z})^2\rtimes\GL(\mathbb{F}_2,2)\cong(\mathbb{Z}/2\mathbb{Z})^2\rtimes\mathfrak{S}_3$, which has order $4\times 3!=24$. On the other hand, the non-linear translation part $b$ is free, which is $(\mathbb{Z}/2\mathbb{Z})^3$ with order $8$. So the stabilizer of $v$ has order $24\times 8=192$. It is well-known that the only order $192$ subgroup of $\mathrm{Aff}(\mathbb{F}_2,3)$ are $(\mathbb{Z}/2\mathbb{Z})^3\rtimes\mathfrak{S}_4$ and $(\mathbb{Z}/2\mathbb{Z})^4\rtimes D_{12}$. There exists an order $4$ element in the stabilizer given by the $4\times4$ upper triangular matrix with $1$ on and above the diagonal, so the stabilizer is $(\mathbb{Z}/2\mathbb{Z})^3\rtimes\mathfrak{S}_4$.
\end{proof}

For a pair $(X,L)\in \widetilde\M$, we constructed an \'etale double cover $Z_L$ of the Persson surface $X$ in \cref{Sec:auxiliary_surfaces}. As discussed in \cref{Rmk:ZL_eigen_interpretation}\cref{Item:1_Rmk:ZL_eigen_interpretation}, there are four K3 surfaces $S_{\widetilde{\chi}_i}$ ($i=1,\ldots,4$) that appear as intermediate covers of the morphism $Z_L\to \mathbf{P}^2$. These surfaces play a crucial role in understanding the $\widetilde{G}^*$-eigenspace decomposition of the anti-invariant part $H^2(X,\mathbb{C})_L^\perp$. 

Recall that each $S_{\widetilde{\chi}_i}$ is a $15$-nodal K3 surface obtained as a double cover of $\mathbf{P}^2$ branched along the six lines $D_{g_j},D_{g_j'}$ ($j\neq i$) given in \cref{Eqn:relabel_g_i_g_i'}. The minimal resolutions of these singular K3 surfaces form a $4$-dimensional family. In the present paper, our arguments rely on these smooth K3 surfaces and their transcendental lattices. Nevertheless, it makes no essential difference to work with K3 surfaces having ADE singularities (for the global Torelli theorem concerning singular K3 surfaces, we refer the reader to \cite[Theorem~5.4]{alexeevengel_lattice}). Now let us recall the relevant results of \cite{matsumoto1992monodromy} (see also \cite{honosoliantakagiyau2020mirrorK3_1}) on the associated period map. 
\begin{notation}\label{Notation:MSY}
    To discuss the period map and the monodromy for the K3 surfaces in \cref{Rmk:ZL_eigen_interpretation}\cref{Item:1_Rmk:ZL_eigen_interpretation}, we will use the following notation.
\renewcommand\labelitemi{\tiny$\bullet$}
\begin{itemize}
    \item $\mathrm{M}$: the generic Picard lattice of the family of K3 surfaces which are minimal resolutions of the double cover of $\mathbf{P}^2$ branched along six lines in general linear position. See \cite[Section~2.1]{matsumoto1992monodromy} for a concrete description of $\mathrm{M}$.
    \item $\mathrm{T}:= U(2)^{\oplus 2}\oplus A_1(-1)^{\oplus 2}$. As shown in \cite[Proposition~2.3.1]{matsumoto1992monodromy}, $\mathrm{T}$ is isometric to the generic transcendental lattice of the above family of K3 surfaces. In particular, $\mathrm{T}=(\mathrm{M})_{\Lambda_{K3}}^\perp$ in the K3 lattice $\Lambda_{K3}=U^{\oplus 3}\oplus E_8(-1)^{\oplus 2}$. (Note that the lattice $\mathrm{T}$ is denoted by $A$ in \cite[Page~10]{matsumoto1992monodromy}.)

    \item $\mathcal{D}_{\mathrm{M}}^0$: a connected component of the period domain $\mathcal{D}_{\mathrm{M}} :=\{\omega\in \mathbf{P}(\mathrm{T}\otimes \mathbb{C})\mid \langle\omega,\omega\rangle_{\mathrm{T}}=0, \langle\omega,\bar{\omega}\rangle_{\mathrm{T}}>0\}$. (In \cite[Section~0.13]{matsumoto1992monodromy}, this type IV Hermitian symmetric domain is denoted by $D$.) 
    
    \item $\Gamma_{\mathrm{T}}^+$: the subgroup of the isometry group $O(\mathrm{T})$ defined in \cite[Section~0.7]{matsumoto1992monodromy}. More concretely, $\Gamma_{\mathrm{T}}^+=\{Y\in\GL(\mathbb{Z},6)\mid Y^tQ_{\mathrm{T}}Y=Q_{\mathrm{T}}, f(Y)>0\}$, where $Q_{\mathrm{T}}$ is the intersection matrix of $\mathrm{T}$ and $f(Y)=(y_{11}+y_{21})(y_{33}+y_{43})-(y_{31}+y_{41})(y_{13}+y_{23})$ for $Y=(y_{ij})$. 
    As discussed in \emph{op.~cit.}, Section~0.13, $\Gamma_{\mathrm{T}}^+$ acts properly discontinuously on $\mathcal{D}_{\mathrm{M}}^0$.
    \item $\Gamma_{\mathrm{T}}(2):=\{Y\in\Gamma_{\mathrm{T}}^+\mid Y\equiv I_6\mod 2\}$. This subgroup of $\Gamma_{\mathrm{T}}^+$ is introduced in \cite[Section~0.7]{matsumoto1992monodromy}, and also acts properly discontinuously on $\mathcal{D}_{\mathrm{M}}^0$ (\emph{op.~cit.}, Section~0.13). By \emph{op.~cit.}, Proposition~2.8.2, there exists an exact sequence 
\begin{equation} \label{Eqn:arithmeticS6}
    1\to \Gamma_{\mathrm{T}}(2)\to \Gamma_{\mathrm{T}}^+\to \mathfrak{S}_6\times \mathbb{Z}/2\mathbb{Z}\to 1.
\end{equation}
\end{itemize}
\end{notation}

Denote the moduli space of six ordered lines in general linear position on $\mathbf{P}^2$ by $\M(\mathbf{P}^2,6)$, as in \cref{Rmk:compare_with_Kollar}. There is a period map \[
    \Phi:\M(\mathbf{P}^2,6)\to \mathcal{D}_{\mathrm{M}}^0/\Gamma_{\mathrm{T}}(2)
\]
sending six lines to the periods of the K3 surface obtained as the minimal resolution of the associated double cover of $\mathbf{P}^2$. According to \cite[Proposition~2.7.3]{matsumoto1992monodromy}, the monodromy group of $\Phi$ is isomorphic to the group $\Gamma_{\mathrm{T}}(2)$. By the global Torelli for K3 surfaces and \emph{op.~cit.}, Proposition~2.10.1, the period map $\Phi$ is injective.

For our purposes, we note that $\mathfrak{S}_6$ acts naturally on $\M(\mathbf{P}^2,6)$ by permuting the six ordered lines. Following the discussion in \cite[Sections~0.11, 0.13]{matsumoto1992monodromy}, this action corresponds to the factor $\mathfrak{S}_6$ in \cref{Eqn:arithmeticS6}. More explicitly, this can be seen via the computation in \emph{op.~cit.}, Section~2.4. In fact, permutations of the six lines induce a homomorphism $\mathfrak{S}_6\to \Gamma_\mathrm{T}^+$. In \emph{op.~cit.}, Lemmas~2.4.1, 2.4.2, and 2.4.3, the images of the transpositions $(12)$, $(34)$, $(46)$, and $(56)$ are explicitly determined. Composing this with the homomorphism $\Gamma_{\mathrm{T}}^+\rightarrow \mathfrak{S}_6\times \mathbb{Z}/2\mathbb{Z}$ from \cref{Eqn:arithmeticS6} and the natural projection $\mathfrak{S}_6\times \mathbb{Z}/2\mathbb{Z}\to \mathfrak{S}_6$, we obtain a homomorphism $\mathfrak{S}_6\to \Gamma_{\mathrm{T}}^+\to \mathfrak{S}_6$. Since any normal subgroup of $\mathfrak{S}_6$ is either trivial, the alternating group $\mathfrak{A}_6$, or the whole group, the kernel of this homomorphism must be one of these three. Using the aforementioned results in \emph{op.~cit.}, Section~2.4, one can readily verify that the kernel is trivial, yielding an isomorphism. Therefore, the factor $\mathfrak{S}_6$ in \cref{Eqn:arithmeticS6} faithfully represents the permutations of the six lines. (The moduli space $\M(\mathbf{P}^2,6)$ also admits an involution $\ast$ corresponding to the factor $\mathbb{Z}/2\mathbb{Z}$ in \cref{Eqn:arithmeticS6}; see \emph{op.~cit.}, Section~0.11.)

Define the subgroup $\widetilde{\Gamma}_{\mathrm{T}}\subset \Gamma_{\mathrm{T}}^+$ as the preimage of the factor $\mathfrak{S}_6$ in \cref{Eqn:arithmeticS6}. In other words, we have the following exact sequence:
\begin{equation}\label{Eqn:tilde_Gamma_T}
1\to \Gamma_{\mathrm{T}}(2)\to \widetilde{\Gamma}_{\mathrm{T}}\to \mathfrak{S}_6\to 1.
\end{equation}
Clearly, $\widetilde{\Gamma}_{\mathrm{T}}$ acts on the type IV domain $\mathcal{D}_{\mathrm{M}}^0$. We  state the following result, which can be viewed as a global Torelli theorem for the moduli space of six unordered lines.

\begin{proposition}\label{Prop:induced_period_map_injective}
    The period map $\Phi:\M(\mathbf{P}^2,6)\to \mathcal{D}_{\mathrm{M}}^0/\Gamma_{\mathrm{T}}(2)$ induces a birational morphism
\[
\varphi: \M(\mathbf{P}^2,6)/\mathfrak{S}_6\to \mathcal{D}_{\mathrm{M}}^0/\widetilde{\Gamma}_{\mathrm{T}},
\]
    which is injective. 
\end{proposition}
\begin{proof}
    The argument closely follows that of \cite[Corollary~4.23, Theorem~4.1]{radu2009n16} (compare also \cite[Proposition~3.18, Theorem~3.23]{gmgz2018quartic}). According to \cite[Proposition~2.10.1]{matsumoto1992monodromy}, the period map $\Phi$ is both injective and dominant. In particular, $\Phi$ is birational. From our preceding analysis, the quotient group $\widetilde{\Gamma}_{\mathrm{T}}/\Gamma_{\mathrm{T}}(2)\cong \mathfrak{S}_6$ corresponds to permuting the six lines. Consequently, $\Phi$ is $\mathfrak{S}_6$-equivariant and thus descends to a well-defined birational morphism $\varphi:\M(\mathbf{P}^2,6)/\mathfrak{S}_6\to \mathcal{D}_{\mathrm{M}}^0/\widetilde{\Gamma}_{\mathrm{T}}$. The injectivity of $\varphi$ follows by descending the injective morphism $\Phi$ to the $\mathfrak{S}_6$-quotient.
\end{proof}

This completes the study of the period map for the K3 surfaces in \cref{Rmk:ZL_eigen_interpretation}\cref{Item:1_Rmk:ZL_eigen_interpretation}. Next, we define a $\widetilde G$-period map for Persson surfaces following \cite{pardini1998period,dolgachev2007moduli}, using the construction from \cref{Sec:auxiliary_surfaces}. Choose a sufficiently general reference point $(X_0,L_0)\in \widetilde{\M}$. Let $Z_{L_0}$ be the \'etale double cover of $X_0$ associated with the $2$-torsion line bundle $L_0$. By \cref{Prop:composition_is_abelian}, $Z_{L_0}$ is an abelian cover of $\mathbf{P}^2$ with Galois group $\widetilde{G}\cong (\mathbb{Z}/2\mathbb{Z})^5$. Consider the covering involution associated with $Z_{L_0}\to X_0$, and denote the anti-invariant part of the torsion-free lattice $H^2_{\mathrm{f}}(Z_{L_0},\mathbb{Z})$ by $H^2(X_0,\mathbb{Z})^\perp_{L_0}$, or simply $H^2(X_0,\mathbb{Z})^\perp$. Setting $H^2(X_0,\mathbb{Q})^\perp:=H^2(X_0,\mathbb{Z})^\perp\otimes \mathbb{Q}$, we note that $H^2(X_0,\mathbb{Q})^\perp$ is naturally a $\widetilde{G}$-representation. Moreover, it carries a polarized $\mathbb{Q}$-Hodge structure with respect to the intersection form.

The $\widetilde{G}^*$-eigenspace decomposition of $H^2(X_0,\mathbb{Q})^\perp$ is described in \cref{Rmk:ZL_eigen_interpretation}, using \cref{prop:eigen}. For our purposes, we observe that the $8$-dimensional subspace of $H^2(X_0,\mathbb{Q})^\perp$ spanned by the primitive cohomology of the eight del Pezzo surfaces of degree $2$ in \cref{Rmk:ZL_eigen_interpretation}\cref{Item:2_Rmk:ZL_eigen_interpretation} consists entirely of algebraic classes. In other words, the $(2,0)$-part of $H^2(X_0,\mathbb{C})^\perp$ is completely contained in the direct sum of the eigenspaces associated with the characters $\widetilde{\chi}_1,\ldots, \widetilde{\chi}_4$ of $\widetilde{G}^*$. Therefore, we define $V_\mathbb{Q}$ to be the $\widetilde{G}$-subrepresentation of $H^2(X_0,\mathbb{Q})^\perp$ consisting of the eigenspaces $H^2(X_0,\mathbb{Q})^{\perp,\widetilde{\chi}_i}$ corresponding to the characters $\widetilde{\chi}_i$ for $i=1,\ldots,4$:
\[
    V_\mathbb{Q}:= \bigoplus_{i=1}^4 H^2(X_0,\mathbb{Q})^{\perp,\widetilde{\chi}_i}.
\]
From \cref{Rmk:ZL_eigen_interpretation}\cref{Item:1_Rmk:ZL_eigen_interpretation} and \cite[Proposition~2.3.1]{matsumoto1992monodromy}, it is not difficult to see that $V_\mathbb{Q}$ is the transcendental part of the Hodge structure on $H^2(X_0,\mathbb{Q})^\perp$ (that is, the minimal sub-Hodge structure whose $(2,0)$-part coincides with the $(2,0)$-part of $H^2(X_0,\mathbb{C})^\perp$). Furthermore, the intersection form on $H^2(X_0,\mathbb{Q})^\perp$, as described in \cref{Lem:lattice_on_H_perp}, induces a polarization on $V_\mathbb{Q}$, which we denote by $Q$. Let $\rho:\widetilde{G}\to \Aut(V_\mathbb{Q},Q)$ be the associated representation.

\begin{notation}\label{Notation:eigen_period_domain}
    To construct our period map for the moduli space $\widetilde{\M}$ defined in \cref{Def:our_moduli_functor}, we introduce the following notation.
\renewcommand\labelitemi{\tiny$\bullet$}
\begin{itemize}
    \item $V_\mathbb{Q}$, $Q$, $\rho$: defined as above. In particular, $\rho$ endows $V_\mathbb{Q}$ with the structure of a $\widetilde{G}$-representation.
    \item  $V_{\mathbb{Q}}^{\widetilde{\chi}_i}$: the eigenspace of $V_\mathbb{Q}$ associated with the character $\widetilde{\chi}_i$, $i=1,\ldots,4$.
    \item $\mathcal{D}$: the period domain of $Q$-polarized Hodge structures of weight $2$ on $V_\mathbb{Q}$ with Hodge numbers $[4,16,4]$.
    \item $\mathcal{D}^\rho$: a connected component of $\{x\in\mathcal{D}\mid\rho(\widetilde g)x=x,\forall \widetilde g\in \widetilde G\}$.
    \item $\mathcal{D}_{\widetilde{\chi}_i}^0$: a connected component of the period domain for the polarized Hodge structures on $(V_{\mathbb{Q}}^{\widetilde{\chi}_i}, Q|_{V_{\mathbb{Q}}^{\widetilde{\chi}_i}})$ of type $[1,4,1]$, $i=1,\ldots,4$.
\end{itemize}
\end{notation}

According to \cite[Equation~(7.6)]{dolgachev2007moduli}, there exists a natural injective map 
\[
    \mathcal{D}^\rho\hookrightarrow\prod_{i=1}^4\mathcal{D}_{\widetilde{\chi}_i}^0.
\]
Since all our characters are real, it is straightforward to verify that the map above is also surjective. Consequently, $\dim \mathcal{D}^\rho=16>8$.
We now show that for each $i=1,\ldots,4$, the domain $\mathcal{D}_{\widetilde{\chi}_i}^0$ is isomorphic to the type IV Hermitian symmetric domain $\mathcal{D}_{\mathrm{M}}^0$ introduced earlier in \cref{Notation:MSY}. 

\begin{lemma}\label{Lem:eigenvsK3}
    For each $i=1,\ldots,4$, there exists an isomorphism 
\[
    (V_{\mathbb{Q}}^{\widetilde{\chi}_i}, \frac{1}{16}Q)\cong \left(\mathrm{T}\otimes \mathbb{Q}, \langle\cdot,\cdot\rangle_{K3}\otimes \mathbb{Q}\right),
\]
    where $\langle\cdot,\cdot\rangle_{\mathrm{K3}}$ denotes the intersection form on the K3 lattice $\Lambda_{K3}=U^{\oplus 3}\oplus E_8(-1)^{\oplus 2}$.
\end{lemma}
\begin{proof}
    Let $(X_0,L_0)\in \widetilde{\M}$ be the sufficiently general reference point chosen above. By definition, $V_{\mathbb{Q}}^{\widetilde{\chi}_i}=H^2(X_0,\mathbb{Q})^{\perp,\widetilde{\chi}_i}$. 
    Recall that $S_{\widetilde{\chi}_i}$ is a $15$-nodal K3 surface obtained as the double cover of $\mathbf{P}^2$ branched along the six lines $D_{\widetilde g_j},D_{\widetilde g_j'}$ ($j\neq i$) given in \cref{Eqn:relabel_g_i_g_i'}. We denote by $\widetilde{S}_{\widetilde{\chi}_i}$ the minimal resolution of $S_{\widetilde{\chi}_i}$, which belongs to the $4$-dimensional family of K3 surfaces studied in \cite{matsumoto1992monodromy}. 

    By \cref{prop:eigen}, the vector space $V^{\widetilde{\chi}_i}_{\mathbb{Q}}$ is isomorphic to the primitive cohomology $H_{\mathrm{prim}}^2(S_{\widetilde{\chi}_i},\mathbb{Q})$ via the pullback map. Using \cite[Corollary~2.1.6]{matsumoto1992monodromy}, we deduce that the algebraic classes in $H^2(\widetilde{S}_{\widetilde{\chi}_i}, \mathbb{Q})$ are spanned by the exceptional curves arising from the $15$ intersection points of the six lines $D_{\widetilde g_j},D_{\widetilde g_j'}$ ($j\neq i$), together with the pullback of a line in $\mathbf{P}^2$ not passing through these intersection points. It follows that the $\mathbb{Q}$-vector space $H_{\mathrm{prim}}^2(S_{\widetilde{\chi}_i},\mathbb{Q})$ is isomorphic to the transcendental component $\mathrm{T}\otimes \mathbb{Q}$ of $H^2(\widetilde{S}_{\widetilde{\chi}_i}, \mathbb{Q})$ via the linear map induced by the minimal resolution $\widetilde{S}_{\widetilde{\chi}_i}\to S_{\widetilde{\chi}_i}$. Finally, the claim concerning the bilinear forms holds since $Z_{L_0}$ is a $(\mathbb{Z}/2\mathbb{Z})^4$-cover of $S_{\widetilde{\chi}_i}$ as shown in \cref{Eqn:partial_cover_tower}.
\end{proof}

To introduce the $\widetilde{G}$-period map, it remains to specify a discrete subgroup $\Gamma_\rho$. For each $i=1,\ldots,4$, let $\Gamma_i \subset \Aut(V_{\mathbb{Q}}^{\widetilde{\chi}_i}, Q)$ be the discrete subgroup corresponding to $\widetilde{\Gamma}_{\mathrm{T}}$ from \cref{Eqn:tilde_Gamma_T} under the isomorphism established in \cref{Lem:eigenvsK3}. Recalling the orthogonal decomposition $V_\mathbb{Q}= \bigoplus_{i=1}^4 V_{\mathbb{Q}}^{\widetilde{\chi}_i}$, we define the subgroup $\Gamma_\rho < \Aut(V_\mathbb{Q},Q)$ to be the direct product 
\begin{equation}\label{Eqn:gamma_rho}
    \Gamma_\rho:= \prod_{i=1}^4\Gamma_i, 
\end{equation}
which acts block-diagonally on $V_\mathbb{Q}$. Following \cite[Section~6]{pardini1998period} and \cite[Section~7]{dolgachev2007moduli}, we consider the following $\widetilde{G}$-period map in \cref{Def:G-period_map}. According to \cite[Section~7]{dolgachev2007moduli} (see also \cite[Pages~67-68]{ggk2012mt}), the monodromy representation commutes with the representation $\rho$. On the other hand, the moduli space of smooth Persson surfaces is given by $\M=\M(\mathbf{P}^2,8)/G_{\Stab}$, where $G_{\Stab}\cong\mathrm{Aff}(\mathbb{F}_2,3)$ acts by permuting the labels of the branch divisors but fixes $\chi_0$ (see \cref{Def:Persson} and \cref{Thm:smoothness_stack}). This implies that the ordering of the branch lines may change, but their support is preserved up to the action of $\PGL(\mathbb{C},3)$. Together with \cref{Prop:induced_period_map_injective}, this analysis shows that the discrete subgroup $\Gamma_\rho$ contains the image of the monodromy representation.

\begin{definition}[{\cite[Section~6]{pardini1998period} and \cite[Section~7]{dolgachev2007moduli}}]\label{Def:G-period_map}
The \emph{$\widetilde G$-period map} (or \emph{eigenperiod map}) for Persson surfaces is defined as
\begin{align*}
    \mathcal{P}:\widetilde{\M}&\rightarrow \mathcal{D}^\rho/\Gamma_\rho,\\
    [(X,L)]&\mapsto \Gamma_\rho\text{-orbit of } \phi_\mathbb{C}^{-1}(\bigoplus_{i=1}^4 H^{2,0}(X,\mathbb{C})^{\perp,\widetilde{\chi}_i})
\end{align*}
where $\phi:V_\mathbb{Q}\to \bigoplus_{i=1}^4 H^2(X,\mathbb{Q})^{\perp,\widetilde{\chi}_i}$ is a $\rho$-marking in the sense of \cite[Section~7]{dolgachev2007moduli} (that is, an isomorphism of $\widetilde{G}$-representations).
\end{definition}

Before stating and proving the main result of this section, we establish the following lemma.

\begin{lemma}\label{Lem:4times6_imply_8}
    Let $D=\sum_{i=1}^4 D_{\widetilde g_i}+D_{\widetilde g_i'}$ and $D'=\sum_{i=1}^4 D'_{\widetilde g_i}+D'_{\widetilde g_i'}$ be two generic $8$-line arrangements on $\mathbf{P}^2$.
    Let $\sigma_i\in\PGL(\mathbb{C},3)$, $i=1,\ldots, 4$, such that
\[
    \sigma_i(\sum_{j\neq i}D_{\widetilde g_j}+D_{\widetilde g_j'})=\sum_{j\neq i}D'_{\widetilde g_j}+D'_{\widetilde g_j'},
\]
    i.e., each $\sigma_i$ maps the support of the three pairs of lines of $D$ to the corresponding one of $D'$. Then there exists $\sigma\in\PGL(\mathbb{C},3)$ such that $\sigma(D)=D'$. Therefore, generically, the four K3 surfaces $S_{\widetilde{\chi}_i}$, $i=1,\ldots,4$ uniquely determine the support of the eight lines, up to the action of $\PGL(\mathbb{C},3)$.
\end{lemma}
\begin{proof}
    If $\sigma_1$ maps the support of $D_{\widetilde g_1}+ D_{\widetilde g_1'}$ to the support of $D_{\widetilde g_1}'+ D_{\widetilde g_1'}'$, then one can take $\sigma=\sigma_1$. Then also $\sigma_i=\sigma$, $i=2,3,4$, at least supportwisely. Otherwise, assume that the support of $\sigma_1(D_{\widetilde g_1}+ D_{\widetilde g_1'})$ is different from the support of $D_{\widetilde g_1}'+ D_{\widetilde g_1'}'$, then $\sigma_2(\sum_{i=3,4}D_{\widetilde g_i}+D_{\widetilde g_i'})$ consists of four lines of $\sum_{i\neq 2} D_{\widetilde g_i}'+D_{\widetilde g_i'}'$. There are only finitely many possible such $\sigma_2$ since $\dim\PGL(\mathbb{C},3)=8=4\times 2$. One also requires that $\sigma_2$ maps $D_{\widetilde g_1}+D_{\widetilde g_1'}$ to the other two lines of $\sum_{i\neq 2} D_{\widetilde g_i}'+D_{\widetilde g_i'}'$. However, for a generic line arrangement, such $\sigma_2$ does not exist.
\end{proof}

\begin{theorem}\label{Thm:degree_2_periodmap}
    Generically, the $\widetilde{G}$-period map $\mathcal{P}:\widetilde{\M}\rightarrow \mathcal{D}^\rho/\Gamma_\rho$ introduced in \cref{Def:G-period_map} has degree $2$.
\end{theorem}
\begin{proof}
    Consider the composition 
\[
     \widetilde{\M}\stackrel{\mathcal{P}}{\rightarrow} \mathcal{D}^\rho/\Gamma_\rho\xrightarrow{\cong}\prod_{i=1}^4\mathcal{D}_{\widetilde{\chi}_i}^0/\Gamma_i,
\]
    where the second arrow is the natural map from \cite[Equation~(7.6)]{dolgachev2007moduli}. Its bijectivity is guaranteed by our choice of $\Gamma_\rho$ and \cite[Page~69]{dolgachev2007moduli}. For a Persson surface $X$ endowed with a $2$-torsion line bundle $L$, according to \cref{Prop:composition_is_abelian}, the associated \'etale double cover $Z_L$ is also a $\widetilde{G}$-cover of $\mathbf{P}^2$. In particular, the branch divisors of $Z_L\to \mathbf{P}^2$ are precisely the same eight lines as those of $X\to \mathbf{P}^2$, which we label as in \cref{Eqn:relabel_g_i_g_i'_with_tilde} and \cref{Eqn:relabel_g_i_g_i'} respectively. By \cref{Lem:eigenvsK3} and the arguments in \cite[Theorem~3.5]{pearlstein2019generic}, one readily verifies that the composite map assigns to $(X,L)$ the periods of the intermediate K3 surfaces $S_{\widetilde{\chi}_i}$, $i=1,\ldots,4$ in \cref{Rmk:ZL_eigen_interpretation}\ref{Item:1_Rmk:ZL_eigen_interpretation} appearing in the covering tower of $Z_L$. Applying \cref{Prop:induced_period_map_injective}, we recover the six unordered lines $\sum_{j\neq i}(D_{\widetilde{g}_j}+D_{\widetilde{g}_j'})$ for each $i=1,\ldots,4$, up to $\PGL(\mathbb{C},3)$-equivalence. Then by \cref{Lem:4times6_imply_8}, this uniquely determines the support of the eight branch lines. 
    
    To specify the building data, and in particular the branch data, for the Persson surface $X$, we must assign elements of $G$ as labels to these eight branch lines. Recall from \cref{Def:Persson} and \cref{Thm:smoothness_stack} that the moduli space of smooth Persson surfaces is $\M=\M(\mathbf{P}^2,8)/G_{\Stab}$, where $G_{\Stab}\cong\mathrm{Aff}(\mathbb{F}_2,3)$ acts by permuting the branch divisor labels while fixing $\chi_0$. Let $s_i$ denote the transposition of the labels $D_{g_i}$ and $D_{g_i'}$, and define $s_I=\prod_{i\in I}s_i$. Direct calculation shows that $s_i\not\in\mathrm{Aff}(\mathbb{F}_2,3)$ but $s_{ij}\in\mathrm{Aff}(\mathbb{F}_2,3)$. Consequently, $s_{1234}\in\mathrm{Aff}(\mathbb{F}_2,3)$ whereas $s_{ijk}\not\in\mathrm{Aff}(\mathbb{F}_2,3)$. Because the action of $\mathfrak{S}_4$ permuting the four pairs in \cref{Eqn:relabel_g_i_g_i'} also lies within $\mathrm{Aff}(\mathbb{F}_2,3)$, the natural action of $(\mathbb{Z}/2\mathbb{Z})^4\rtimes\mathfrak{S}_4$ on these pairs yields exactly two non-isomorphic Persson surfaces. These two surfaces differ by a single transposition $s_i$ of the labels within a pair $\{D_{g_i},D_{g_i'}\}$. In other words, exactly two distinct Persson surfaces map to the same quadruple of K3 surfaces $S_{\widetilde{\chi}_i}$ ($i=1,\ldots,4$). Since the torsion line bundle $L$ is uniquely recovered via \cref{Rmk:torsion_L_vs_four_pairs}\ref{Item:1_Rmk:torsion_L_vs_four_pairs}, we conclude that the $\widetilde{G}$-period map $\mathcal{P}$ is generically of degree $2$.
\end{proof}

\begin{remark}\label{Rmk:notinjective}
\begin{enumerate}[wide]
    \item We point out that the degree of the $\widetilde{G}$-period map actually coincides with the index of the subgroup $(\mathbb{Z}/2\mathbb{Z})^3\rtimes\mathfrak{S}_4$ of $(\mathbb{Z}/2\mathbb{Z})^4\rtimes\mathfrak{S}_4$ appearing in \cref{Prop:tM_MP2_8_quotient}, where the latter includes all possible permutations of the four pairs in \cref{Eqn:relabel_g_i_g_i'}. The appearance of two non-isomorphic Persson surfaces with the same support of branch data can also be seen via the condition in \cref{Rmk:ab_cover_uniqueness}. In fact, an Enriques surface in \cref{Rmk:surfaces_in_the_middle}\cref{Item:2_rmk_surfaces_in_middle} admits two different bidouble covers with the same branch data, as its Picard group is not $2$-torsion free.
    
    \item One may manually construct a degree $1$ $\widetilde{G}$-period map by enlarging the period domain. For instance, instead of using a single \'etale double cover $Z_L$, one can simultaneously consider all seven double covers arising from $\Pic(X)[2]$. By doing so, the ambiguity caused by the transposition $s_i$ in the proof of \cref{Thm:degree_2_periodmap} would be eliminated. However, we do not pursue this modified $\widetilde{G}$-period map here, since such a construction looks not very natural.

    \item It is unclear whether the period map without using the $\widetilde{G}$-action is also generically of degree $2$. In other words, we do not know if one can still naturally obtain the $\widetilde{G}^*$-eigenspace decomposition without keeping track of the $\widetilde{G}$-action. The same question can also be posed for the generic global Torelli theorem of the special Horikawa surfaces studied by Pearlstein--Zhang \cite{pearlstein2019generic}.    
\end{enumerate}
\end{remark}

\bibliographystyle{alpha}

\Adresses

\end{document}